\documentclass[12pt,reqno]{amsart}

\title[Optimal transport]
{A saddle-point approach to the {M}onge-{K}antorovich optimal
transport problem}
\author{Christian L\'eonard}
\date{May 09}

\usepackage{amssymb, amsmath, amsfonts, latexsym, enumerate}

\newtheorem{theorem}[equation]{Theorem}
\newtheorem{lemma}[equation]{Lemma}
\newtheorem{proposition}[equation]{Proposition}

\newtheorem{counterexamples}[equation]{Counterexamples}

\newtheorem{definition}[equation]{Definition}
\newtheorem{definitions}[equation]{Definitions}

\theoremstyle{remark}
\newtheorem{remark}[equation]{Remark}
\newtheorem{remarks}[equation]{Remarks}

\numberwithin{equation}{section}

\newcommand{\R}{\mathbb{R}}

\newcommand{\1}{\textbf{1}}

\newcommand{\dom}{\mathrm{dom\,}}
\newcommand{\icordom}{\mathrm{icordom\,}}
\newcommand{\cl}{\mathrm{cl\,}}
\newcommand{\cv}{\mathrm{cv\,}}

\newcommand{\lsc}{lower semicontinuous}

\newcommand{\supp}{\mathrm{supp}\,}

\newcommand{\boulette}[1]{$\bullet$\ Proof of #1.}
\newcommand{\Boulette}[1]{\par\medskip\noindent $\bullet$\ Proof of #1.}

\newcommand\seq[2]{(#1_#2)_{#2\ge1}}
\newcommand\Lim[1]{\lim_{#1\rightarrow\infty}}
\newcommand\Liminf[1]{\liminf_{#1\rightarrow\infty}}

\newcommand\Glim[1]{\Gamma\textrm{-}\lim_{#1\rightarrow\infty}}

\newcommand{\ttimes}{\!\times\!}

\newcommand{\diffdom}{\mathrm{dom}\partial}

\newcommand{\Uo}{\mathcal{U}_o}
\newcommand{\Yo}{\mathcal{Y}_o}
\newcommand{\Lo}{\mathcal{L}_o}
\newcommand{\Xo}{\mathcal{X}_o}
\newcommand{\UU}{\mathcal{U}}
\newcommand{\LL}{\mathcal{L}}
\newcommand{\YY}{\mathcal{Y}}
\newcommand{\XX}{\mathcal{X}}
\newcommand{\Ls}{\mathcal{L}^*}
\newcommand{\Xs}{\mathcal{X}^*}

\newcommand{\AB}{{A\!\times\! B}}
\newcommand{\PA}{\mathcal{P}_A}
\newcommand{\PB}{\mathcal{P}_B}
\newcommand{\PAB}{\mathcal{P}_{AB}}
\newcommand{\CAB}{C_{AB}}
\newcommand{\CAs}{C_A^*}
\newcommand{\CBs}{C_B^*}
\newcommand{\CABs}{\CAB^*}
\newcommand{\Pc}{\mathcal{P}_c}

\newcommand{\s}{\mathcal{S}}
\newcommand{\PS}{\mathcal{P}_{\s}}

\newcommand{\Uc}{U_{c}}
\newcommand{\Cs}{C_{|\s}}

\newcommand{\F}{\Phi}
\newcommand{\Fo}{\Phi_o}
\newcommand{\Fs}{\Phi^*}
\newcommand{\Fos}{\Phi_o^*}
\newcommand{\Fb}{\overline{\Phi}}
\newcommand{\La}{\Lambda}
\newcommand{\Las}{\Lambda^*}
\newcommand{\Lab}{\overline{\Lambda}}
\newcommand{\fog}{f\oplus g}

\newcommand{\NF}[1]{|#1|_\F}
\newcommand{\NL}[1]{|#1|_\La}
\newcommand{\Nc}[1]{\|#1\|_c}
\newcommand{\Ncs}[1]{\|#1\|_c^*}

\newcommand{\HF}{$(H_\F)$}
\newcommand{\HFi}{$(H_{\F1})$}
\newcommand{\HFii}{$(H_{\F2})$}
\newcommand{\HFiii}{$(H_{\F3})$}
\newcommand{\HT}{$(H_T)$}

\newcommand{\yx}{\langle y,x\rangle}
\newcommand{\ul}{\langle u,\ell\rangle}

\newcommand{\lt}{\tilde{\ell}}
\newcommand{\qb}{\bar{q}}

\newcommand{\ob}{\bar\omega}
\newcommand{\IAB}{\int_{\AB}}
\newcommand{\IS}{\int_{\s}}
\newcommand{\uS}{u_{|\s}}
\newcommand{\etat}{\widetilde{\eta}}
\newcommand{\etab}{\bar{\eta}}
\newcommand{\etah}{\widehat{\eta}}

\newcommand{\MK}{\textrm{M\!K}}

\newcommand{\Kb}{\overline{\textrm{K}}}
\newcommand{\Kt}{\widetilde{\textrm{K}}}
\newcommand{\mn}{(\mu,\nu)}
\newcommand{\mnc}{(\mu,\nu,c)}
\newcommand{\Pmn}{P(\mu,\nu)}
\newcommand{\Pcmn}{P(\mu,\nu,c)}
\newcommand{\Qcmn}{Q(\mu,\nu,c)}

\begin{document}

% ************** page de garde ******************************

 \address{Modal-X, Universit\'e Paris Ouest. B\^at. G, 200 av. de la R\'epublique. 92001 Nanterre, France}
 \email{christian.leonard@u-paris10.fr}
 \keywords{Convex optimization, saddle-point, conjugate duality, optimal transport}
 \subjclass[2000]{46N10, 49J45, 28A35}

\begin{abstract}
The Monge-Kantorovich problem is revisited by means of a variant
of the saddle-point method without appealing to $c$-conjugates. A
new abstract characterization of the optimal plans is obtained in
the case where the cost function takes infinite values. It leads
us to new explicit sufficient and necessary optimality conditions.
As by-products, we obtain a new proof of the well-known
Kantorovich dual equality and an improvement of the convergence of
the minimizing sequences.
\end{abstract}

\maketitle

\tableofcontents

% ************* corps du texte ****************************

\section{Introduction}\label{sec:introduction}

The Monge-Kantorovich problem is revisited by means of a variant
of the saddle-point method derived in \cite{Leo07a}, without
appealing to $c$-conjugates. A new abstract characterization of
the optimal plans is obtained (Theorem \ref{res-20}) in the case
where the cost function takes infinite values. It leads us to new
explicit sufficient and necessary conditions of optimality which
are stated in Theorems \ref{res-05} and \ref{res-73}. As
by-products, we obtain a new proof of the well-known Kantorovich
dual equality and an improvement of the convergence of the
minimizing sequences.

\subsection*{The Monge-Kantorovich transport problem}

Let us take $A$ and $B$ two Polish  (separable complete metric)
spaces furnished with their respective Borel $\sigma$-fields, a
\lsc\ (cost) function $c:\AB \to [0,\infty]$ which may take
infinite values and two probability measures $\mu\in\PA$ and
$\nu\in\PB$ on $A$ and $B.$ We denote $\PA, \PB$ and $\PAB$ the
sets of all Borel probability measures on $A,$ $B$ and $\AB.$ The
Monge-Kantorovich problem is
\begin{equation}\label{MK}
    \textsl{minimize } \pi\in\PAB\mapsto \IAB c(a,b)\,\pi(dadb)
    \textsl{ subject to }\pi\in \Pmn \tag{\MK}
\end{equation}
where $\Pmn$ is the set of all $\pi\in\PAB$ with prescribed
marginals $\pi_A=\mu$ on $A$ and $\pi_B=\nu$ on $B.$ Note that $c$
is measurable since it is \lsc\ and the integral $\IAB
c\,d\pi\in[0,\infty]$ is well-defined since $c\geq 0.$
\\
For a general account on this active field of research, see the
books of S.~Rachev and L.~R\"uschendorf \cite{RacRus} and
C.~Villani \cite{Vill03,Vill09}.
\\
The subset $\{c=\infty\}$ of $\AB$ is a set of forbidden
transitions. Optimal transport on the Wiener space
\cite{FeyUstu04a} and on configuration spaces \cite{p-Dec06,
DJS08} provide natural infinite dimensional settings where $c$
takes infinite values.

Let us denote $C_A,$ $C_B$ and $\CAB$ the spaces of all continuous
bounded functions on $A,$ $B$ and $\AB.$ The Kantorovich
maximization problem:
\begin{equation}\label{K}
\left\{
\begin{split}
   &\textsl{maximize }\int_A f\,d\mu+\int_B g\,d\nu \textsl{ for all
    } f, g \textsl{ such that }\\
    & f\in C_A, g\in C_B\textsl{ and } f\oplus g\leq
    c\\
\end{split}
\right. \tag{K}
\end{equation}
is the basic dual problem of \eqref{MK}. Here and below, we denote
$ f\oplus g(a,b)= f(a)+ g(b).$ Under our assumptions, we have
\begin{equation}
    \inf\eqref{MK}=\sup\eqref{K}\in [0,\infty]
    \label{eq-04}
    \end{equation}
which is called the Kantorovich dual equality. For a proof of this
well known result, see \cite[Thm 5.10]{Vill09} for instance.

\begin{definitions}[Plans]\label{def-02}\
\begin{enumerate}
    \item Any probability measure in $\Pmn$ is called a transport plan, or
    shorter: a \emph{plan}.
    \item One says that  $\pi\in \Pmn$ is a \emph{finite plan} if
$\IAB c\,d\pi<\infty.$ The set of all finite plans is denoted by
$\Pcmn.$
    \item One says that  $\pi$ is an \emph{optimal plan} if it is
    a finite plan and it
minimizes $\gamma\mapsto\IAB c\,d\gamma$ on $\Pmn.$
\end{enumerate}
\end{definitions}

It is well-known that there exists at least an optimal plan if and
only if $\Pcmn$ is nonempty; this will be found again in Theorem
\ref{res-MK}. Definition \ref{def-02}-3 throws away the
uninteresting case where $\IAB c\,d\pi=\infty$ for all $\pi\in
\Pmn.$
\\
Since it is a convex but not a \emph{strictly} convex problem,
infinitely many optimal plans may exist.

Recently, M.~Beiglb\"ock, M.~Goldstern, G.~Maresh,
W.~Schachermayer and J.~Teichman \cite{ST09,BGMS08,BS09} have
improved previous optimality criteria in several directions.
Before stating some of their results, let us introduce the notion
of strongly $c$-cyclically monotone plan.

Clearly, there is no reason for the dual equality \eqref{eq-04} to
be attained at \emph{continuous} functions $f$ and $g.$ Suppose
instead that there exist two $[-\infty,\infty)$-valued integrable
functions $f\in \mathcal{L}^1(\mu)$ and $g\in\mathcal{L}^1(\nu)$
such that $\fog\le c$ everywhere and $\sup\{\int_A u\,d\mu+\int_B
v\,d\nu; u\in\mathcal{L}^1(\mu),v\in\mathcal{L}^1(\nu),u\oplus
v\le c\}=\int_A f\,d\mu+\int_B g\,d\nu.$ Let $\pi$ be an optimal
plan. We have
$$\sup\eqref{K} \le
\sup_{u\in\mathcal{L}^1(\mu),v\in\mathcal{L}^1(\nu),u\oplus v\le
c}\int_A u\,d\mu+\int_B v\,d\nu=\int \fog\,d\pi\le\int
c\,d\pi=\inf\eqref{MK}$$ and by \eqref{eq-04}, this is a series of
equalities. In particular, the couple $(f,g)$ is an
\emph{integrable dual optimizer}, $\int (c-\fog)\,d\pi=0$ and
since $c-\fog\ge0,$ we see that $c=\fog,$ $\pi$-almost everywhere.
This leads naturally to the following notion.

\begin{definition}[Strongly $c$-cyclically monotone plan,  \cite{ST09}]\label{def-04}
A transport plan $\pi\in \Pmn$ is called \emph{strongly
$c$-cyclically monotone} if there exist $[-\infty,\infty)$-valued
measurable  functions $ f$ on $A$ and
    $ g$ on  $B$  such that
    \begin{equation}\label{eq-75}
    \left\{\begin{array}{l}
       f\oplus g\leq c\quad \textrm{everywhere} \\
        f\oplus g= c\quad \pi\textrm{-almost everywhere.} \\
    \end{array}\right.
    \end{equation}
\end{definition}

In the whole paper, measurable functions and sets are intended to
be Borel measurable.

Note that it is not required in this definition that $f$ and $g$
are integrable. In fact, one can find examples where there is an
optimal plan but no integrable dual optimizer, see \cite[Example
4.5]{BS09}. Without integrability, one cannot write $\int_A
f\,d\mu+\int_B g\,d\nu$ anymore. Nevertheless, the following
result allows us to extend the notion of dual optimizer to
measurable functions $(f,g)$ as in Definition \ref{def-04}.

\begin{lemma}\cite[Lemma 1.1]{BS09}\label{res-03}
Let $c:\AB\to[0,\infty]$ be measurable, $f$ and $g$ be
$[-\infty,\infty)$-valued measurable functions on $A$ and  $B$
respectively such that $f\oplus g\le c$ everywhere. Then, for any
$\pi,\tilde\pi\in\Pcmn,$
\begin{equation*}
    \IAB f\oplus g\,d\pi=\IAB f\oplus g\,d\tilde\pi\in[-\infty,\infty).
\end{equation*}
\end{lemma}
With this lemma in hand, we are allowed to denote
\begin{equation}\label{eq-05}
    \IAB f\oplus g\,d\mnc:=\IAB f\oplus g\,d\pi,\quad\pi\in\Pcmn
\end{equation}
this common value.  A natural extension of the dual problem
\eqref{K} is

\begin{equation}\label{Kt}
     \left\{
   \begin{split}
     &\textsl{maximize } \IAB \fog\,d\mnc
     \textsl{ for all } f\in [-\infty,\infty)^{A},
     g\in[-\infty,\infty)^{B}\\
     &f,g  \textsl{ measurable such that } \fog\leq c \textsl{ everywhere.}
    \end{split}
    \right.\tag{$\Kt$}
\end{equation}
Of course, $\sup\eqref{K}\le\sup\eqref{Kt}\le\inf\eqref{MK}$ so
that \eqref{eq-04} implies
\begin{equation}
    \sup\eqref{Kt}=\inf\eqref{MK}.
    \label{eq-06}
\end{equation}
A couple of functions $(f,g)$ as in Definition  \ref{def-04} is
clearly an optimizer of \eqref{Kt}. We call it a \emph{measurable
dual optimizer}.

Some usual optimality criteria are expressed in terms of cyclical
$c$-monotonicity.
\begin{definition}[$c$-cyclically monotone plan]
A subset $\Gamma\subset\AB$ is said to be \emph{$c$-cyclically
monotone} if for any integer $n\geq 1$ and any family
$(a_1,b_1),\dots,(a_n,b_n)$ of points in $\Gamma,$
 $
 %\begin{equation*}
    \sum_{i=1}^n c(a_i,b_i)\leq \sum_{i=1}^n c(a_i,b_{i+1})
 %\end{equation*}
 $
with the convention $b_{n+1}=b_1.$
\\
A probability measure $\pi\in\PAB$ is said to be $c$-cyclically
monotone if it is concentrated on a measurable $c$-cyclically
monotone set $\Gamma,$ i.e. $\pi(\Gamma)=1.$
\end{definition}
This notion goes back to the seminal paper \cite{Rus96} by
L.~R\"uschendorf where the standard cyclical monotonicity of
convex functions introduced by Rockafellar has been extended in
view of solving the Monge-Kantorovich problem.
\\
One easily shows that a strongly $c$-cyclically monotone plan is
$c$-cyclically monotone.

The main results of \cite{ST09,BGMS08,BS09} are collected in the
next two theorems.

\begin{theorem}[\cite{BGMS08,BS09}]\label{res-15}
Let $c$ be a measurable nonnegative function such that
\begin{equation}\label{eq-76}
    \mu\otimes\nu(\{c< \infty\})=1.
    \end{equation}
If there exists some $\pi^o\in \Pmn$ such that $\IAB
c\,d\pi^o<\infty,$ then \eqref{eq-06} holds true and for any
$\pi\in \Pmn,$ the following three statements are equivalent:
\begin{enumerate}[(i)]
    \item $\pi$ is an optimal plan;
    \item $\pi$ is $c$-cyclically monotone;
    \item $\pi$ is strongly $c$-cyclically monotone.
\end{enumerate}
\end{theorem}

This result is valid under a very weak regularity condition on $c$
which is only assumed to be Borel measurable. It was first proved
in \cite{ST09} under the assumption that $c$ is a \lsc\
nonnegative finitely-valued function. However, condition
(\ref{eq-76}) is close to the assumption that $c$ is finitely
valued.

The next result is concerned with cost functions $c$ which may
take infinite values.
\begin{theorem}[\cite{AP02,ST09,BGMS08}]\label{res-75}
Let $c$ be a \lsc\ $[0,\infty]$-valued function.
\begin{enumerate}[(a)]
    \item Any optimal plan is  $c$-cyclically monotone.
    \item If there exists some $\pi^o\in \Pmn$ such that $\IAB
    c\,d\pi^o<\infty,$ then any strongly $c$-cyclically monotone plan in $\Pmn$
    is an optimal plan.
\end{enumerate}
Let $c$ be a measurable $[0,\infty]$-valued function.
\begin{enumerate}[(c)]
    \item Every finite $c$-cyclically monotone transport plan is optimal if
there exist a closed set F and a $\mu\otimes\nu$-null set N such
that $\{c=\infty\}=F\cup N.$
\end{enumerate}
\end{theorem}
Statement (a) is proved in L.~Ambrosio and A.~Pratelli's lecture
notes \cite{AP02} and (b) is taken from \cite{ST09}. Statement (c)
is proved in \cite{BGMS08}, it extends recent results of
A.~Pratelli \cite{Pra08}.

\begin{counterexamples}\label{res-06}  Otherwise stated,
$c$ is assumed to be \lsc.
\begin{enumerate}
    \item  $c$ is real valued.
    \begin{enumerate}
        \item \cite[Example 1.3]{BGMS08}  $c$ is not \lsc\ and no optimal plan
         exists.
        \item \cite[Example 4.5]{BS09} $c$ is the squared
        distance and no \emph{integrable} dual optimizers exist.
    \end{enumerate}
    \item  $c$ takes infinite values on a non-null set.
    \begin{enumerate}
         \item \cite[Example 4.1]{BS09} $c$ is not
        \lsc\ and the weak Kantorovich dual equality \eqref{eq-06} fails to
        hold: We have $\sup\eqref{Kt}<\inf\eqref{MK}.$
        \item \cite[Example 3.5]{AP02} A $c$-cyclically
        monotone plan which is not optimal.
        \item \cite[Example 1]{ST09} A $c$-cyclically monotone plan which is not
        strongly $c$-cyclically monotone.
        \item \cite[Example 5.1]{BGMS08} An optimal plan which is not
        strongly $c$-cyclically monotone.
        \item \cite[Example 4.2]{BS09} An optimal plan which is not
        strongly $c$-cyclically monotone with $c$ continuous.
    \end{enumerate}
\end{enumerate}
\end{counterexamples}

It appears that Theorem \ref{res-15} is the best possible result
under the assumption \eqref{eq-76} but that its extension to the
general case where $c$ takes infinite values on a non-null set is
still open. The Counterexample 2-b tells us that one must drop the
notion of $c$-cyclically monotone plan and with the
Counterexamples 2-d and 2-e, one sees  that even the notion of
strongly $c$-cyclically monotone plan is not enough to
characterize optimality.

Because of the dual gap when $c$ is not \lsc\ (Counterexample
2-a), it is reasonable to assume that $c$ is \lsc\ in the general
case. This will be assumed from now on.

\subsection*{The aim of this paper} This paper is aimed at
implementing the saddle-point method for solving \eqref{MK} in the
general case where $c$ might take infinite values, without
appealing to $c$-conjugates, see \cite{Vill03,Vill09}. After all,
\eqref{MK} is a convex minimization problem and one may wonder
what the standard approach based on convex duality could yield. It
appears that an extended version of the saddle-point is necessary
to investigate \eqref{MK} without imposing strong qualifications
of the constraints such as the standard requirement:
$\int_Ac_A\,d\mu+\int_Bc_B\,d\nu<\infty$ where $c\le c_A\oplus
c_B,$ see \cite{Vill03, Vill09}. The present paper relies on an
extension of the standard saddle-point method which has been
developed in \cite{Leo07a}.

\begin{enumerate}[-]
    \item A new proof of the Kantorovich dual equality \eqref{eq-04} is given in Theorem \ref{res-MK} providing at the
same time an improved result about the behavior of the minimizing
sequences of \eqref{MK}.
    \item An abstract characterization of the optimal plans is given in
Theorem \ref{res-20}. It expresses a saddle-point property and
extends the notion of strong $c$-cyclical monotonicity.
    \item It leads easily to a sufficient condition of optimality
    in Theorem \ref{res-05} below which is an improvement of Theorem
    \ref{res-75}-b. Its proof is given in Section \ref{sec-02}.
    \item Finding good necessary conditions for  optimal plans with a
genuinely $[0,\infty]$-valued cost function $c$ is still an open
problem. Theorem \ref{res-73} below, which is a corollary of
Theorem \ref{res-20}, goes one step in this direction. Its proof
is given in Section \ref{sec-01}.
\end{enumerate}

One says that $f$ is $\mu$-measurable if there exists a measurable
set $N$ such that $\mu(N)=0$ and $\1_Nf$ is a measurable function.
A property holds $\Pcmn$-almost everywhere if it holds true
outside a measurable set $N$ such that $\gamma(N)=0,$ for all
$\gamma\in\Pcmn.$ With these definitions in hand, we can state our
sufficient condition of optimality.

\begin{theorem}\label{res-05}
Let $\pi\in\Pcmn$ be any finite plan. If there exist a
$\mu$-measurable function $ f$ on $A$ and a $\nu$-measurable
function $ g$ on  $B$  which satisfy
\begin{equation}\label{eq-08bis}
    \left\{\begin{array}{l}
       f\oplus g\leq c\quad \Pcmn\textrm{-almost everywhere} \\
        f\oplus g= c\quad \pi\textrm{-almost everywhere,} \\
    \end{array}\right.
    \end{equation}
then $\pi$ is optimal.
\end{theorem}

The Counterexamples \ref{res-06}-(d,e) are optimal plans which are
not  strongly $c$-cyclically monotone but they both satisfy the
weaker property \eqref{eq-08bis}. See Subsection \ref{sec-ctex}
for more details.

The following result is our necessary condition of optimality.

\begin{theorem}\label{res-73}
Take any optimal plan $\pi,$  $\epsilon>0$ and $p$ any probability
measure on $\AB$ such that $\IAB c\,dp<\infty.$ Then, there exist
functions $\varphi\in L_1(\pi+p),$ $u$ and $v$ bounded continuous
on $A$ and $B$ respectively and a measurable subset $D\subset\AB$
such that
\begin{enumerate}
    \item
    $\varphi= c,\ \pi$-almost
    everywhere on $\AB\setminus D;$
    \item
    $\int_{D} (1+c)\,d\pi\le\epsilon;$
    \item
    $-c/\epsilon\le\varphi\le c,\ (\pi+p)$-almost everywhere;
    \item
    $-c/\epsilon\le u\oplus v\le c,$ everywhere;
    \item
    $\|\varphi-u\oplus v\|_{L_1(\pi+p)}\le\epsilon.$
\end{enumerate}
\end{theorem}

An important thing to notice in this result is the appearance of
the probability measure $p$ in items (3) and (5). One can read
(3-5) as an approximation of $\fog\le c,$ $(\pi+ p)$-almost
everywhere. Since it is required that $\IAB c\,dp<\infty,$ one can
choose $p$ in $\Pcmn,$ and the properties (1-5) of this theorem
are an approximation of \eqref{eq-08bis} where $\Pcmn$-a.e.\! is
replaced by $(\pi+p)$-a.e.

In \cite{BLS09a}, M. Beiglb\"ock, W. Schachermayer and the author
of the present paper investigate the same problem with a different
approach which is still based on a saddle point method. A
characterization of optimal plans which is more explicit than
Theorem \ref{res-20}  is obtained under some restrictive
assumptions. In particular, the problem of considering
simultaneously all the measures $p\in\Pcmn$ in Theorem
\ref{res-73} is handled by means of a projective limit argument.

\subsection*{Notation}
Let $X$ and $Y$ be topological vector spaces and  $f: X\rightarrow
[-\infty,+\infty]$ be an extended-real valued function.
\begin{enumerate}[-]
    \item The algebraic dual space of $X$ is $X^{\ast},$ the topological
dual space of $X$ is $X'.$
    \item The topology of $X$ weakened by $Y$ is $\sigma(X,Y)$ and one
writes $\langle X,Y\rangle$ to specify that $X$ and $Y$ are in
separating duality.
    \item The convex conjugate of $f$ with respect to $\langle X,Y\rangle$ is
$f^*(y)=\sup_{x\in X}\{\langle x,y\rangle -f(x)\}\in
[-\infty,+\infty],$ $y\in Y.$
    \item The subdifferential of $f$ at $x$ with
respect to $\langle X,Y\rangle$ is $\partial_Y f(x)=\{y\in Y;
f(x+\xi)\geq f(x)+\langle y,\xi\rangle, \forall
 \xi\in X\}.$ If no confusion occurs, one writes
$\partial f(x).$
    \item The effective domain of $f$ is $\dom f=\{x\in X; f(x)<\infty\}.$
    \item One denotes $\icordom f$ the intrinsic core of the effective domain of $f$
$\dom f.$ Recall that the intrinsic core of a subset $A$ of a
vector space is
    $
\mathrm{icor\,} A=\{x\in A; \forall x'\in\mathrm{aff\,}A, \exists
t>0, [x,x+t(x'-x)[\subset A\}
    $
where $\mathrm{aff\,}A$ is the affine space spanned by $A.$

    \item The domain of $\partial f$ is $\diffdom f=\{x\in X;\partial f(x)\not=\emptyset\}.$
    \item The indicator function of a subset $A$ of $X$ is defined by
    \begin{equation*}
    \iota_A(x)=\left\{%
\begin{array}{ll}
    0, & \hbox{if }x\in A \\
    +\infty, & \hbox{otherwise} \\
\end{array}%
\right.,\quad x\in X.
   \end{equation*}
    \item The support function of $A\subset X$ is $\iota_A^*(y)=\sup_{x\in
A}\langle x,y\rangle,$ $y\in Y.$
    \item The Dirac measure at $a$ is denoted $\delta_a.$
\end{enumerate}

\subsection*{Acknowledgements}
The author thanks Stefano Bianchini for very useful comments.

\section{The abstract convex minimization problem}
\label{sec:abstractpb} The Monge-Kantorovich problem is a
particular instance of an abstract convex minimization problem
which is solved in \cite{Leo07a} by means of an extension of the
saddle-point method.  This extension allows to remove standard
topological restrictions on the constraint sets (the so-called
constraint qualifications) at the price of solving an arising new
problem. Namely, one has to compute a \lsc\ convex extension of
the convex conjugate of the objective function. This may be a
rather difficult task in some situations, but it is immediate in
the case of the Monge-Kantorovich problem.
\\
Let us recall the main results of \cite{Leo07a}.

\subsection*{The minimization problem}

Let $\Uo,$ $\Yo$ be two vector spaces, $$\Fo:\Uo\to[0,\infty]$$ an
extended-nonnegative  function on  $\Uo,$ and $$T_o^*:\Yo\to\Uo$$
a linear mapping from $\Yo$ to $\Uo.$ Throughout this section, it
is assumed that
\begin{itemize}
    \item[\HF]
1-\quad $\Fo: \Uo\rightarrow [0,+\infty]$ is convex and $\Fo(0)=0$\\
2-\quad $\forall u\in\Uo, \exists \alpha >0, \Fo(\alpha u)<\infty$\\
3-\quad $\forall u\in\Uo, u\not=0,\exists t\in\R, \Fo(tu)>0$
    \item[\HT]
1-\quad $T_o^\ast (\Yo)\subset\Uo$\\
2-\quad $\mathrm{ker\ }T_o^\ast =\{0\}$
\end{itemize}
Let $\Lo$ and $\Xo$ be the \emph{algebraic} dual vector spaces of
$\Uo$ and $\Yo.$ The convex conjugate of $\Fo$ with respect to the
dual pairing $\langle\Uo,\Lo\rangle$ is
\begin{equation*}
    \Fos(\ell)=\sup_{u\in\Uo}\{\ul-\Fo(u)\}\in[0,\infty],\quad \ell\in\Lo
    \end{equation*}
and the adjoint of $T_o^*$ is $T_o:\Lo\to\Xo$ defined for all
$\ell\in\Lo$ by $\langle y,T_o\ell\rangle_{\Yo,\Xo}=\langle T_o^*
y,\ell\rangle_{\Uo,\Lo}.$ We obtain the diagram
\begin{equation}
 \begin{array}{ccc}
\Big\langle\ \Uo & , & \Lo \ \Big\rangle \\
T^\ast_o  \Big\uparrow & & \Big\downarrow
 T_o
\\
\Big\langle\ \Yo & , & \Xo\ \Big\rangle
\end{array}
\tag{Diagram 0}
\end{equation}
For each $x\in\Xo,$ the optimization problem to be considered is
\begin{equation*}
    \textsl{minimize } \Fos(\ell) \quad\textsl{subject to}\quad T_o\ell=x,\qquad \ell\in\Lo   \tag{$P_o^x$}
\end{equation*}

\subsection*{A dual problem}

Let us define
\begin{eqnarray*}
  \La_o(y) &:=& \Fo(T_o^* y),\quad y\in\Yo \\
  \NF u &:=& \inf\left\{a>0; \max\Big(\Fo(-u/a),\Fo(u/a)\Big)\le1\right\},\quad u\in\Uo \\
  \NL y &:=& \inf\left\{a>0;
  \max\Big(\La_o(-y/a),\La_o(y/a)\Big)\le1\right\},\quad y\in\Yo
\end{eqnarray*}
Under our assumptions, $\NF\cdot$ and $\NL\cdot$ are norms. One
introduces
\begin{enumerate}[-]
    \item $\UU$ and $\YY$ the completions of $(\Uo,\NF\cdot)$ and
$(\Yo,\NL\cdot);$
    \item The \emph{topological} dual spaces of $\UU$ and
$\YY$ are $\UU'=\LL$ and $\YY'=\XX;$
    \item The \emph{algebraic} dual spaces of $\LL$ and $\XX$ are
    denoted by $\Ls$ and $\Xs;$
    \item $T$ is the restriction of $T_o$ to $\LL\subset\Lo;$
    \item $\Fs$ is the restriction of $\Fos$ to $\LL\subset\Lo.$
\end{enumerate}
It is proved in \cite{Leo07a} that $T\LL\subset\XX.$ This allows
one to define the algebraic adjoint $T^*:\Ls\to\Xs.$ We have the
following diagram
\begin{equation}
 \begin{array}{ccc}
\Big\langle\ \LL & , & \Ls \ \Big\rangle \\
T \Big\downarrow & & \Big\uparrow
 T^\ast
\\
\Big\langle\ \XX & , & \Xs\ \Big\rangle
\end{array}
\tag{Diagram 1}
\end{equation}
Also consider
\begin{equation}
\left\{
\begin{array}{lcll}
 \Fb(\zeta) &:=& \sup_{\ell\in\LL}\{\langle\zeta,\ell\rangle-\Fs(\ell)\},&\quad \zeta\in\Ls \\
  \Lab(\omega) &:=& \Fb(T^*\omega),&\quad\omega\in\Xs.\\
\end{array}
  \right.
\end{equation}
The maximization problem
\begin{equation*}
    \textsl{maximize } \langle\omega,x\rangle -\Lab(\omega), \quad\omega\in\Xs   \tag{$\overline{D}^x$}
\end{equation*}
is the dual problem associated with
\begin{equation*}
    \textsl{minimize } \Fs(\ell) \quad\textsl{subject to}\quad T\ell=x,\qquad \ell\in\LL   \tag{$P^x$}
\end{equation*}
with respect to the usual Fenchel perturbation.

\subsection*{Statements of the abstract results}

Let us introduce
\begin{equation*}
    \Las(x):=\sup_{y\in\Yo}\{\yx-\La_o(y)\},\quad x\in\XX
\end{equation*}
the convex conjugate of $\La_o+\iota_{\Yo}$ with respect to the
dual pairing $\langle \YY,\XX\rangle.$

\begin{theorem}[Primal attainment and dual equality, \cite{Leo07a}]\label{xP3}
Assume that \HF\ and \HT\ hold.
\begin{enumerate}[(a)]
    \item For all $x$ in $\XX,$ we have the dual equality
        \begin{equation}\label{xped}
            \inf\{\Fs(\ell); \ell\in\LL, T\ell=x\}=\inf(P^x_o)=\inf(P^x)=\Las(x)\in
            [0,\infty].
        \end{equation}
        Moreover,
        $\Las$ is $\sigma(\XX,\YY)$-inf-compact.
    \item If $x\not\in\XX,$ $(P^x_o)$ admits no solution.
    Otherwise, $P^x_o$ and $P^x$  admit the same (possibly empty) set of solutions.
    \item If in addition $$x\in\dom\Las,$$ then $P^x_o$ (or equivalently $P^x$) is attained in
    $\LL.$ Moreover, any minimizing sequence for $P^x_o$ has $\sigma(\LL,\UU)$-cluster
    points and every such cluster point solves $P^x_o$.
\end{enumerate}
\end{theorem}
Denote the ``algebraic'' subdifferentials
\begin{eqnarray*}
  \partial_{\LL}\Fb(\zeta_o) &=& \{\ell\in\LL;\Fb(\zeta_o+\zeta)\ge \Fb(\zeta_o)+\langle\zeta,\ell\rangle,\forall\zeta\in\Ls\},\quad \zeta_o\in\Ls \\
  \partial_{\Xs}\Las(x_o) &=& \{\omega\in\Xs; \Las(x_o+x)\ge \Las(x_o)+\langle \omega,x\rangle, \forall
  x\in\XX\}, \quad x_o\in\XX
\end{eqnarray*}
and define
 $$
\diffdom\Las=\{x\in\XX;\partial_{\Xs}\Las(x)\not=\emptyset\}
 $$
the subset of the constraint specifiers $x\in\XX$ such that
$\partial_{\Xs}\Las(x)$ is not empty.

\begin{theorem}[Dual attainment and representation, \cite{Leo07a}]\label{T3b}
Let us assume that \HF\ and \HT\ hold.
\begin{enumerate}[(1)]
    \item
For any  $\ell\in\LL$ and $\omega\in\Xs,$
 \begin{equation}\label{eq-97}
    \left\{\begin{array}{l}
     T\ell=x\\
     \ell\in\partial_{\LL}\Fb(T^*\omega) \\
    \end{array}\right.
\end{equation}
is equivalent to
\begin{equation*}
    \left\{%
\begin{array}{l}
    \hbox{$\ell$ is a solution to $P^{x}$,} \\
    \hbox{$\omega$ is a solution to  $\overline{D}^x$ and} \\
    \hbox{the dual equality $\inf(P^{x})=\Las(x)$ holds.} \\
\end{array}%
\right.
\end{equation*}

    \item Suppose that in addition the constraint qualification
\begin{equation}\label{xCQbis}
    x\in\diffdom\Las,
\end{equation}
is satisfied. Then, the primal problem $P^{x}$ is attained in
$\LL,$ the dual problem $\overline{D}^x$ is attained in
    $\XX^*$ and every couple of solutions $(\ell,\omega)$  to $P^{x}$ and $\overline{D}^x$ satisfies
    (\ref{eq-97}).
\end{enumerate}
\end{theorem}

By the geometric version of Hahn-Banach theorem, $\icordom\Las:$
the intrinsic core of $\dom\Las,$ is included in $\diffdom\Las.$
But, as will be seen at Remark \ref{rem-04} below, the
Monge-Kantorovich problem provides us with a situation where
$\icordom\Las$ is empty. This is one of the main difficulties to
be overcome when applying the saddle-point method to solve
\eqref{MK}.

It is a well-known result of convex conjugacy that the
representation formula
        \begin{equation}\label{xeq-110}
    \ell\in\partial_{\LL}\Fb(T^*\omega)
\end{equation}
 is equivalent to
 \begin{equation*}
    T^*\omega\in\partial_{\Ls}\Fs(\ell)
\end{equation*}
 and also equivalent to Fenchel's identity
\begin{equation}\label{xeq-107}
    \Fs(\ell)+\Fb(T^*\omega)=\langle \omega,T\ell\rangle.
\end{equation}

\begin{proposition}[\cite{Leo07a}]\label{res-01}
Assume that \HF\ and \HT\ hold. Any solution $\omega\in\Xs$ of
$\overline{D}^x$ shares the following properties
\begin{itemize}
 \item[(a)] $\omega$ is in the $\sigma(\XX^*,\XX)$-closure of $\dom\La;$
 \item[(b)] $T^\ast \omega$ is in the $\sigma(\LL^*,\LL)$-closure of
 $T^\ast(\dom\La).$
\end{itemize}
If in addition the level sets of $\Fo$ are $\NF\cdot$-bounded,
then
\begin{itemize}
 \item[(a')] $\omega$ is in $\YY''.$ More precisely, it is in the $\sigma(\YY'',\XX)$-closure of $\dom\La;$
 \item[(b')] $T^\ast \omega$ is in $\UU''.$ More precisely, it is in the $\sigma(\UU'',\LL)$-closure of
 $T^\ast(\dom\La)$
\end{itemize}
where $\YY''$ and $\UU''$ are the topological bidual spaces of
$(\UU,\NF\cdot)$ and $(\YY,\NL\cdot).$
\end{proposition}

\section{Kantorovich dual equality and existence of optimal plans}\label{sec:transport}

We apply the results of Section \ref{sec:abstractpb} to the
Monge-Kantorovich problem \eqref{MK}.

\subsection{Statement of the results}

The set of all probability measures $\pi$ on $\AB$ such that $\IAB
c\,d\pi<\infty$ is denoted $\Pc.$ Hence,
    $
    \Pcmn=\Pc\cap P\mn.
    $
It is immediate to see that
\begin{equation}
    \sup\eqref{K}\le\sup(\Kt)\le\inf\eqref{MK}
    \label{eq-95}
\end{equation}
where for the last inequality, we state $\inf\emptyset=+\infty$ as
usual.
\\
 In the next theorem, $\Pc$ will be
endowed with the weak topology $\sigma(\Pc,\mathcal{C}_c)$ where
$\mathcal{C}_c$ is the space of all continuous functions $u$ on
$\AB$ such that
$$
\Lim t \sup\{|u(x)|/(1+c(x)); x\in\AB, |u(x)|\ge t\}=0
$$
with the convention that $\sup\emptyset=0$ which implies that the
space $\CAB$ of bounded continuous functions on $\AB$ satisfies
    $
\CAB\subset \mathcal{C}_c.
    $

\begin{theorem}[Dual equality and
primal attainment]\label{res-MK}\
\begin{enumerate}
 \item The dual equality for \eqref{MK} is
\begin{equation}
    \inf\eqref{MK}=\sup\eqref{K}=\sup\eqref{Kt}\in [0,\infty].
    \label{eq-kanto}
\end{equation}
 \item Assume that there exists
some $\pi^o$ in $\Pmn$ such that $\IAB c\,d\pi^o<\infty.$ Then:
    \begin{enumerate}
            \item There is at least one optimal plan;
            \item Any minimizing sequence is relatively compact for the topology
            $\sigma(\Pc, \mathcal{C}_c)$ and all its cluster points are optimal plans.
   \end{enumerate}
\end{enumerate}
\end{theorem}

\begin{proof}
It directly follows from Theorem \ref{xP3} with Proposition
\ref{res-61}-(3) and Proposition \ref{res-02} which are stated and
proved below. About statement (1), note that the dual equality
(\ref{xped}) is $\inf\eqref{MK}=\sup\eqref{K}$ and conclude with
(\ref{eq-95}).
\end{proof}

Except for the appearance of the dual problem ($\Kt$) and for the
statement (2-b) with the topology $\sigma(\Pc, \mathcal{C}_c)$,
this result is well-known. The dual equality (\ref{eq-kanto}) is
the Kantorovich dual equality. The proof of Theorem \ref{res-MK}
will be an opportunity to make precise the abstract material
$\Fo,$ $\Uo,$ $T_o$ and so on in terms of the Monge-Kantorovich
problem.

\subsection{The spaces and functions}

Let us particularize the spaces and functions of Section
\ref{sec:abstractpb} which are necessary for solving the
Monge-Kantorovich problem.

\subsubsection*{Description of $\Uo,$ $\Lo,$ $\Fo$ and $\Fos$}

Define
\begin{equation*}
    \s:=\{(a,b)\in \AB;c(a,b)<\infty\}
\end{equation*}
the effective domain of $c$ and consider the equivalence relation
on $\CAB$ defined for any $u,v$ in $\CAB$ by
\begin{equation*}
    u\sim v\Leftrightarrow u_{|\s}=v_{|\s},
\end{equation*}
which means that $u$ and $v$ match on $\s.$ Take
\begin{equation*}
    \Uo:=\CAB/\sim
\end{equation*}
the set of all equivalence classes and
\begin{equation*}
    \Fo(u):=\iota_{\Gamma}(u)=\left\{%
\begin{array}{ll}
    0, & \hbox{if }u\le c \\
    +\infty, & \hbox{otherwise} \\
\end{array}%
\right.    ,\quad u\in\Uo
\end{equation*}
the convex indicator function of
\begin{equation*}
    \Gamma=\{u\in\Uo;u\le c\}.
\end{equation*}
Assumption \HFi\ is obviously satisfied and it is easy to see that
\HFiii\ also holds. Note that, without factorizing by the
equivalence relation in the definition of $\Uo,$ \HFiii\ would
fail if $c$ is infinite on some nonempty open set. For \HFii\ to
be satisfied, assume that
\begin{equation*}
    c\ge1.
\end{equation*}
Doing this, one doesn't loose any generality since changing $c$
into $c+1$ raises $\inf\eqref{MK}$ by $+1$ without changing the
minimizers.
\\
The algebraic dual space $\Lo=\Uo^*$ of $\Uo$ is identified as a
subset of $\CAB^*$ as follows. Any $\ell\in\CAB^*$ is in $\Lo$ if
and only if for all $u,v\in\CAB$ such that $u_{|\s}=v_{|\s},$ we
have $\ell(u)=\ell(v).$ One shows at Proposition \ref{res-16}
below that any $\ell\in\CAB^*$ which is in $\Lo$ has its support
included in the closure $\cl\s$ of $\s$ in the sense of the
following
\begin{definition}\label{def-01}
For any linear form $\ell$ on the space of all (possibly
unbounded) continuous functions on $\AB,$ we define the
\emph{support} of $\ell$ as the subset of all $(a,b)\in\AB$ such
that for any open neighborhood $G$ of $(a,b),$ there exists some
function $u$ in $\CAB$ satisfying $\{u\not=0\}\subset G$ and
$\langle u,\ell\rangle \not=0.$ It is denoted by
$\mathrm{supp\,}\ell.$
\end{definition}
\begin{proposition}\label{res-16}
For any $\ell\in\Lo,$ we have $\supp\ell\subset\cl\s.$
\end{proposition}
\begin{proof}
As $\ell\in\Lo,$ for all $u\in \CAB$ such that $u_{|\s}=0,$ we
have $\langle u,\ell\rangle=0.$ Take $x\not\in\cl\s.$ There exists
an open neighborhood $G$ of $x$ such that $G\cap\s=\emptyset.$ For
all $u\in\CAB$ such that $\{u\not=0\}\subset G,$ we have
$u_{|\s}=0$ so that $\langle u,\ell\rangle=0.$ This proves that
$x\not\in\supp\ell.$
\end{proof}

Denote $\Cs$ the space of the restrictions $u_{|\s}$ of all the
functions $u$ in $\CAB.$

\begin{remark}
Beware, $\Cs$ is smaller than the space $C_\s$ of all bounded and
continuous functions on $\s.$ It consists of the functions in
$C_\s$ which admit a continuous and bounded extension to $\AB.$
\end{remark}
Identifying
\begin{equation}\label{eq-03}
    \Uo\cong \Cs,
\end{equation}
one sees with Proposition \ref{res-16} that for all $ \ell\in\Lo,$
\begin{equation*}
    \Fos(\ell)
    :=\sup\{\langle u,\ell\rangle; u\in\Cs,u\le c\}.
\end{equation*}

\subsubsection*{Description of $\Yo,$ $\Xo,$ $T_o^*$ and $T_o$}
Consider
\begin{equation*}
    \Yo:=(C_A\times C_B)/\sim
\end{equation*}
where $C_A,$ $C_B$ are the spaces of bounded continuous functions
on $A,$ $B$ and $\sim$ is a new equivalence relation on $C_A\times
C_B$ defined for all $ f, f'\in C_A,$ $ g, g'\in C_B$ by
$$
( f, g)\sim( f', g')\stackrel{\mathrm{def}}\Leftrightarrow
 f\oplus g\sim f'\oplus g'\Leftrightarrow
 f\oplus g_{|\s}= f'\oplus g'_{|\s}
$$
The algebraic dual space $\Xo=\Yo^*$ of $\Yo$ is identified as a
subset of $\CAs\times\CBs$ as follows. Any
$(k_1,k_2)\in\CAs\times\CBs$ is in $\Xo$ if and only if for all $
f, f'\in C_A,$ $ g, g'\in C_B$  such that $ f\oplus g_{|\s}=
f'\oplus g'_{|\s},$ we have $\langle  f,k_1\rangle+\langle
 g,k_2\rangle=\langle  f',k_1\rangle+\langle  g',k_2\rangle.$
\\
 It is immediate to
see that the operator $T_o^*$ defined by
\begin{equation*}
    T_o^*( f, g):= f\oplus g_{|\s},\quad ( f, g)\in\Yo
\end{equation*}
satisfies the assumptions \HT. In particular, one sees that
$T_o^*(\Yo)\subset\Cs.$
\\
Define the $A$ and $B$-marginal projections $\ell_A\in\CAs$ and
$\ell_B\in\CBs$ of any $\ell\in\CABs$ by
$$
\langle f,\ell_A\rangle:=\langle f\otimes 1,\ell\rangle\quad
\textrm{and}\quad \langle g,\ell_B\rangle:=\langle 1\otimes
g,\ell\rangle
$$
for all $ f\in C_A$ and all $ g\in C_B.$ Since for all $( f,
g)\in\Yo,$
 $\langle T_o^\ast( f, g),\ell\rangle_{\Cs,\Lo}
 =\langle f\oplus g,\ell\rangle_{\Cs,\Lo}
 =\langle f,\ell_A\rangle+\langle g,\ell_B\rangle
 =\langle( f, g),(\ell_A,\ell_B)\rangle,$
one obtains that for all $\ell\in\Lo\subset\CAB^*,$
\begin{equation*}
    T_o(\ell)=(\ell_A,\ell_B)\in\Xo\subset C_A^*\times C_B^*.
\end{equation*}

\subsubsection*{Description of $\UU$ and $\YY$}

For each $u\in C_{|\s},$
$\Fo^{\pm}(u):=\max(\Fo(u),\Fo(-u))=\iota_{\{|u|\le c\}}.$ Hence,
 $\NF u$ is the norm
 $$
\Nc u:=\sup_\s|u|/c,\quad u\in\Cs
 $$
Recall that it is assumed that $c\ge1.$ Let us denote the Banach
space
$$\UU:=U_c$$ which is the completion of $\Cs$ with respect to the
norm $ \Nc\cdot.$

\begin{proposition}\label{res-61}\
\begin{enumerate}
    \item $\Uc$ is a space of functions on $\s$ which are
    continuous in restriction to each level set $\{c\le \alpha\},$
    $\alpha\ge1.$ In particular, any $u\in\Uc$ is a measurable function.
    \item Even if $c$ is finite everywhere, $\Uc$ may contain
    functions which are not continuous.
    \item $\mathcal{C}_c\subset\Uc.$
\end{enumerate}
\end{proposition}
\begin{proof}
 \boulette{(1)}
Let $\seq un$ be a $\Nc\cdot$-Cauchy sequence in $\Cs.$ For all
$x\in\s,$  $\Lim n u_n(x):=u(x)$ exists since $(u_n(x))_{n\ge1}$
is a Cauchy sequence in $\R.$ For each $\epsilon>0,$ and there are
large enough $n,m$ such that
$\sup_{x\in\s}|u_n(x)-u_m(x)|/c(x)\le\epsilon$ for all $x\in\s.$
Letting $m$ tend to infinity leads us to
$|u_n(x)-u(x)|/c(x)\le\epsilon$ for all $x\in\s,$ which gives us
$$
\Lim n \Nc{u_n-u}=0
$$
where the definition $\|v\|_c:=\sup_\s|v|/c$ still holds for any
function $v$ on $\s.$ As for any $x,y\in\s,$
\begin{equation*}
   \frac{u(x)-u(y)}{c(x)+c(y)}
   \le 2\|u-u_n\|_c+\frac{u_n(x)-u_n(y)}{c(x)+c(y)}
   \le 2\|u-u_n\|_c+ |u_n(x)-u_n(y)|,
\end{equation*}
the announced continuity result follows from the above limit. The
measurability statement follows from this continuity result and
the measurability of $c.$

 \Boulette{(2)}
To see this, take
\begin{enumerate}[$\ast$]
    \item $A=B=\R;$
    \item $c(a,b)=1+1/a^2$ for any $(a,b)$ with $a\not=0,$ and
    $c(0,b)=1$ for all real $b;$
    \item $0\le u_n\le1$ with $u_n(a,b)=0$ if $a\le -1/n,$
    $u_n(a,b)=1$ if $a\ge 1/n$ and $u_n(0,b)=1/2$ for all real
    $b.$
\end{enumerate}
Then, $\seq un$ admits the limit $u$ in $\Uc$ with $u(a,b)=0$ if
$a<0,$ $u(a,b)=1$ if $a>0$ and $u(0,b)=1/2.$

 \Boulette{(3)}
By the very definition of $\mathcal{C}_c,$ one sees that for any
$u\in \mathcal{C}_c,$ defining $u_n:=(-n)\vee u\wedge n\in\CAB,$
the sequence $\seq un$ converges to $u$ in $\Uc.$
\end{proof}

The norm on $\Yo$ is given by
\begin{equation*}
    |( f, g)|_\La:=\| f\oplus g\|_c=\sup_\s| f\oplus g|/c,\quad
    ( f, g)\in\Yo.
\end{equation*}
Identifying
\begin{equation*}
    \Yo\cong (C_A\oplus C_B)_{|\s}
\end{equation*}
by means of $( f, g)\in\Yo\mapsto f\oplus g_{|\s}\in(C_A\oplus
C_B)_{|\s}\subset\Cs,$ the $|\cdot|_\La$-completion $\YY$ of $\Yo$
is identified with the $\|\cdot\|_c$-closure of $(C_A\oplus
C_B)_{|\s}$ in $\Uc.$

\subsubsection*{Description of $\LL,$ $\XX,$ $\La$ and $\Las$}

The topological dual space of $\Uc$
\begin{equation*}
    \LL:= \Uc'
\end{equation*}
is equipped with the dual norm
\begin{equation*}
    \|\ell\|_c^*=\sup\{\langle u,\ell\rangle; u\in \Uc,
    \|u\|_c\le1\},\quad \ell\in \Uc'.
\end{equation*}
The function $\La$ is given by
\begin{equation*}
    \La( f, g)=\iota_{\{ f\oplus  g\le c\}},\quad
     f\oplus g\in (C_A\oplus C_B)_{|\s}
\end{equation*}
and the corresponding norm is
\begin{equation*}
    |( f, g)|_\La=\| f\oplus  g\|_c,\quad
     f\oplus g\in (C_A\oplus C_B)_{|\s}.
\end{equation*}
The topological dual space of $\YY$ is
\begin{equation*}
    \XX=\{(k_1,k_2)\in\Xo; |(k_1,k_2)|_\La^*<\infty \}
\end{equation*}
with
\begin{eqnarray*}
  |(k_1,k_2)|_\La^*
  &=& \sup\{\langle f,k_1\rangle+\langle g,k_2\rangle; ( f, g)\in\Yo,\| f\oplus g\|_c\le1\} \\
  &=& \inf\{\|\ell\|_c^*;\ell\in \Uc':\ell_A=k_1,\ell_B=k_2\}
\end{eqnarray*}
where this last equality is a direct consequence of the dual
equality (\ref{xped}). Hence,  $k=(k_1,k_2)\in\CAs\times\CBs$
belongs to $\XX$ if and only if there exists some $\ell(k)\in \LL$
such that $\langle f,k_1\rangle+\langle g,k_2\rangle=\langle
f\oplus g,\ell(k)\rangle$ for all $( f, g)\in\Yo$ and $\Lim
n\langle u_n,\ell(k)\rangle=0$ for any sequence $\seq un$ in
$\CAB$ such that $\Lim n \|u_n\|_c=0.$
\\
The function $\Las$ is given by
\begin{equation*}
    \Las(k_1,k_2)=\sup\{\langle f,k_1\rangle+\langle g,k_2\rangle; ( f, g)\in\Yo, f\oplus g\le
    c\}, \quad k=(k_1,k_2)\in\XX.
\end{equation*}

\subsubsection*{Description of $\Xs$ and $T^*$}

Seeing $\XX$ as a subspace of $\CAs\times\CBs,$ $\Xs$ is a
subspace of $(\CAs\times\CBs)^*.$ By the axiom of choice, there
exists a subspace $\mathcal{Z}$ of $\CAs\times\CBs$ in direct sum
with $\XX:$ $\CAs\times\CBs=\XX\oplus\mathcal{Z}.$ To any
$\omega\in\Xs,$ one associates its extension $\bar{\omega}$ to
$\CAs\times\CBs$ characterized by $\bar{\omega}_{|\mathcal{Z}}=0$
and $\bar{\omega}_{|\XX}=\omega.$ This permits us to define the
marginal projections $\omega_A$ and $\omega_B$ and their tensor
sum $\omega_A\oplus\omega_B\in \Uc^{\prime\ast}$ as follows. For
all $(k_1,k_2)\in\XX,$
\begin{eqnarray*}
  \langle\omega,(k_1,k_2)\rangle_{\Xs,\XX}
  &=&  \langle \bar{\omega},(k_1,k_2)\rangle_{(\CAs\times\CBs)^*,\CAs\times\CBs}\\
  &=&  \langle \bar{\omega},(k_1,0)+(0,k_2)\rangle_{(\CAs\times\CBs)^*,\CAs\times\CBs}\\
  &:=& \langle \omega_A,k_1\rangle_{C_A^{**},\CAs}
        +\langle \omega_B,k_2\rangle_{C_B^{**},\CBs}\\
  &:=& \langle \omega_A\oplus\omega_B,\ell(k)\rangle_{\Uc^{\prime\ast},\Uc'}
\end{eqnarray*}
For any $\omega\in\Xs$ and $\ell\in \Uc',$ $ \langle
T^*\omega,\ell\rangle_{\Uc^{\prime\ast},\Uc'}
  =\langle \omega,(\ell_A,\ell_B)\rangle_{\Xs,\XX}
  := \langle \omega_A\oplus\omega_B,\ell\rangle_{\Uc^{\prime\ast},\Uc'}$
  which means that
\begin{equation}\label{eq-92}
    T^*\omega=\omega_A\oplus\omega_B\in \Uc^{\prime\ast},\quad
    \omega\in\Xs.
\end{equation}

\subsection{The connection with (MK)}
The connection with the Monge-Kantorovich problem is given at
Proposition \ref{res-02} below.  The modified primal problem is
\begin{equation}
    \textsl{minimize } \Fs(\ell) \textsl{ subject to
    }\ell_A=\mu \textsl{ and } \ell_B=\nu,\quad \ell\in \Uc'
    \tag{$P_1$}
\end{equation}
where
\begin{equation*}
  \Fs(\ell)=\sup\{\langle u,\ell\rangle; u\in \Uc, u\leq
    c\},\quad \ell\in \Uc'.
\end{equation*}

\begin{definition}\
\begin{enumerate}[(a)]
    \item One says that $\ell\in \Uc'$ \emph{acts as a probability measure}
if there exists $\tilde{\ell}\in\PAB$ such that
$\supp\tilde{\ell}\subset\cl\s$ and for all $u\in \CAB,$ $\langle
\uS,\ell\rangle=\IS u\,d\tilde{\ell}.$  In this case, we write:
$\ell\in\PS.$
    \item One says that $\ell\in \Uc'$ stands in $\Pc$ if there exists
$\tilde{\ell}\in\Pc$ such that $\supp\tilde{\ell}\subset\cl\s$ and
for all $u\in \Uc,$ $\langle \uS,\ell\rangle=\IS
u\,d\tilde{\ell}.$  In this case, we write: $\ell\in\Pc.$
\end{enumerate}
\end{definition}
Of course, if there exists $\tilde{\ell}$ satisfying (a), it
belongs to $\Pc$ and is unique since any probability measure on a
metric space is determined by its values on the continuous bounded
functions. This explains why the notation $\ell\in\Pc$ in (b) is
not misleading.
\\
Note also that any probability measure $\tilde{\ell}\in\Pc$ has a
support included in $\cl\s.$ Since $\AB$ is a metric space, for
any $\ell\in\Pc$ acting as a measure, $\supp\ell$ in the sense of
Definition \ref{def-01} matches with the usual support of the
measure $\tilde{\ell}.$

\begin{proposition}\label{res-02}
For all $\ell\in \Uc',$
\begin{enumerate}[(a)]
    \item $\Fs(\ell)<\infty\Rightarrow \ell\geq 0,$
    \item $\Fs(\ell)<\infty\Rightarrow \supp\ell\subset\cl\s,$
    \item $[\ell\geq 0,\supp\ell\subset\cl\s, \ell_A=\mu \textsl{ and } \ell_B=\nu]\Rightarrow
    \ell\in\PS$ and
    \item for all $\ell\in\PS,$ $\Fs(\ell)=\IS c\,d\ell\in [0,\infty].$
\end{enumerate}
It follows that  $\dom\Fs\subset\Pc$ and the problems \eqref{MK}
and \emph{$(P_1)$} share the same values and the same minimizers.
\end{proposition}

 \proof
 Clearly, the last statement follows from the first part of the
 proposition.

\Boulette{(a)} Suppose that $\ell\in \Uc'$ is not in the
nonnegative\ cone. This means that there exists $u_o\in \Uc$ such
that $u_o\geq 0$ and $\langle u_o,\ell\rangle<0.$ Since  $u_o$
satisfies $\lambda u_o\leq 0\leq c$ for all $\lambda<0,$ we have
$\Fs(\ell)\geq \sup_{\lambda<0}\langle \lambda
u_o,\ell\rangle=+\infty.$ Hence, $\Fs(\ell)<\infty$ implies that
$\ell\geq 0$ and one can restrict our attention to the
nonnegative\ $\ell$'s.

\Boulette{(b)} Suppose ad absurdum that
$\supp\ell\not\subset\cl\s.$ Then, there exists a  nonnegative\
function $u_o\in\CAB$ such that $\{u_o>0\}\cap\s = \emptyset$ and
$\langle u_o,\ell\rangle>0.$ As $\lambda u_o\leq
c_{|\AB\setminus\s}\equiv\infty$ for all $\lambda>0,$
$\Fs(\ell)\geq \sup_{\lambda>0}\langle
 \lambda u_o,\ell\rangle=+\infty.$

\Boulette{(c)} Let us take $\ell\geq 0$ such that
$\supp\ell\subset\cl\s,$ $\ell_A=\mu$ and $\ell_B=\nu.$ It is
clear that $\langle 1,\ell\rangle=1.$ It remains to check that for
any $\ell\in \Uc'$
\begin{equation}\label{eq-18}
    [\ell\geq 0, \supp\ell\subset\cl\s, \ell_A=\mu \textsl{ and } \ell_B=\nu]\Rightarrow
    \ell \textrm{ is }\sigma\textrm{-additive,}
\end{equation}
rather than only additive. Since $\AB$ is a metric space, one can
apply an extension of the construction of Daniell's integrals
(\cite{Neveu}, Proposition II.7.2) to see that $\ell$ acts as a
measure if and only if for any  decreasing\  sequence $(u_n)$ of
continuous functions such that $0\leq u_n\leq 1$ for all $n$ and
$\lim_{n\rightarrow\infty} u_n=0$ pointwise, we have
$\lim_{n\rightarrow\infty}\langle {u_n},\ell\rangle=0.$ This
insures the $\sigma$-additivity of $\ell.$
\\
Unfortunately, this pointwise convergence of $(u_n)$ is weaker
than the uniform convergence with respect to which any $\ell\in
\Uc'$ is continuous. Except if $\AB$ is compact, since in this
special case, any  decreasing\  sequence of continuous functions
which converges pointwise to zero also converges uniformly on the
compact space $\cl\s.$
\\
So far, we have only used the fact that $\AB$ is a metric space.
We now rely on the Polishness of $A$ and $B$ to get rid of this
compactness restriction. It is known that any probability measure
$P$ on a Polish space $X$ is tight (i.e.\! a Radon measure): for
all $\epsilon>0,$ there exists a compact set $K_\epsilon\subset X$
such that $P(X\setminus K_\epsilon)\leq \epsilon$ (\cite{Neveu},
Proposition II.7.3). As in addition a Polish space is completely
regular, there exists a continuous function $f_\epsilon$ with a
compact support such  that $0\leq f_\epsilon\leq 1$ and $\int_X
(1-f_\epsilon)\,dP\leq \epsilon.$ This is true in particular for
the probability measures $\mu\in\PA$ and $\nu\in\PB$ which specify
the constraint in \eqref{MK}. Hence, there exist $ f_\epsilon\in
C_A$ and $ g_\epsilon\in C_B$ with compact supports such that
$0\leq
 f_\epsilon, g_\epsilon\leq 1$ and $0\leq\int_A
(1- f_\epsilon)\,d\mu,\int_B (1- g_\epsilon)\,d\nu \leq \epsilon.$
It follows with the elementary fact: $0\le1-st\le2-s-t$ for all
$0\le s,t\le1,$ that any nonnegative $\ell\in \Uc'$ with
$\ell_A=\mu$ and $\ell_B=\nu$ satisfies $0\leq \langle (1-
f_\epsilon\otimes g_\epsilon),\ell\rangle\leq 2\epsilon.$ With the
following easy estimate $0\leq \langle u_n,\ell\rangle\leq
2\epsilon +\langle u_n( f_\epsilon\otimes g_\epsilon),\ell\rangle
$ and the compactness of the support of $ f_\epsilon\otimes
g_\epsilon,$ one concludes that $\lim_{n\rightarrow\infty}\langle
u_n,\ell\rangle=0$ which proves (\ref{eq-18}).

\Boulette{(d)} Let us take $\ell\in\PS.$ As $c$ is bounded below
and \lsc, there exists a sequence $\seq cn$ in $\CAB$ such that
$0\le c_n\le c$ converges pointwise and increasingly to $c.$
Therefore,
\begin{equation}\label{eq-90}
    \int_\s c\,d\ell\stackrel{(a)}=\Lim n \int_\s c_n\,d\ell\stackrel{(b)}=\Lim n\langle
    c_n,\ell\rangle\le\Fs(\ell)
\end{equation}
where equality (a) holds by monotone convergence and equality (b)
holds since $c_n$ belongs to $\CAB.$
\\
Let us show the converse inequality: $\Fs(\ell)\le \int_\s
c\,d\ell.$ Let us first assume that $\ell\in\dom\Fs.$ It is proved
at Lemma \ref{res-60} below that for any $u\ge0$ in $\Uc,$
$\langle u,\ell\rangle=\int_\s u\,d\ell.$ It follows that
$\Fs(\ell)=\sup\{\langle u,\ell\rangle; u\in \Uc, u\leq c\}\le
\int_\s c\,d\ell.$
\\
If $\Fs(\ell)=\infty,$ there exists a sequence $\seq un$ in $\CAB$
such that $0\le u_n\le c$ and $\sup_n\int_\s u_n\,d\ell=\infty.$
Therefore, $\int c\,d\ell\ge \sup_n\int u_n\,d\ell=\infty.$
\endproof

\begin{lemma}\label{res-60}
For any $u\ge0$ in $\Uc$ and any $\ell\in\PS\cap\Uc'$ such that
$\Fs(\ell)<\infty,$ we have $\langle u,\ell\rangle=\int_\s
u\,d\ell.$
\end{lemma}

\begin{proof}
There exists a sequence $\seq un$ in $\CAB$ such that $\Lim n
u_n=u$ in $\Uc.$ As $\ell$ belongs to $\Uc',$
\begin{equation}\label{eq-91}
    \Lim n \int_\s u_n\,d\ell=\Lim n \langle
    u_n,\ell\rangle=\langle u,\ell\rangle.
\end{equation}
On the other hand, $|u_n-u|\le\|u_n-u\|_c c$ implies that
$|u_n|\le [\|u\|_c+\|u_n-u\|_c]c.$ Hence, for some large enough
$n_o,$
\begin{equation*}
    |u_n|\le (1+\|u\|_c) c,\quad \forall n\ge n_o.
\end{equation*}
Together with (\ref{eq-90}), the assumption that
$\Fs(\ell)<\infty$ and the dominated convergence theorem, this
entails $\Lim n\int_\s u_n\,d\ell=\int_\s u\,d\ell.$ This and
(\ref{eq-91}) lead us to $\langle u,\ell\rangle=\int_\s u\,d\ell$
which is the desired result.
\end{proof}

\section{An abstract characterization of optimality}

The abstract characterization of optimality is stated in Theorem
\ref{res-20}. It will allow us to obtain as corollaries, an
explicit sufficient condition in Theorem \ref{res-05} and an
explicit necessary condition in Theorem \ref{res-73}.
 To prove it, one has to compute the extension
$\Fb.$ As it is the greatest convex $\sigma(\Uc^{\prime
*},\Uc')$-\lsc\ extension of $\Phi=\iota_\Gamma$ and
$\Gamma=\{u\in\Uo;u\leq c\}$ is a convex subset of $\Uo,$ we have
\begin{equation}\label{eq-72}
    \Fb(\xi)=\iota_{\overline{\Gamma}}(\xi),\quad \xi\in \Uc^{\prime *}
\end{equation}
where $\overline{\Gamma}$ is the $\sigma(\Uc^{\prime
*},\Uc')$-closure of $\Gamma.$ By (\ref{eq-92}), this gives
\begin{equation*}
    \overline{\La}(\omega)=\iota_{\overline{\Gamma}}(\omega_A\oplus\omega_B),\quad
    \omega \in\Xs
\end{equation*}\label{Kb}
and the extended problem ($\overline{D}^x$) is
\begin{equation}
   \textsl{maximize }
   \langle\omega_A,\mu\rangle+\langle\omega_B,\nu\rangle ,\quad
   \omega\in\XX^* \textsl{ such that } \omega_A\oplus\omega_B\in\overline{\Gamma}
   \tag{$\Kb$}
\end{equation}
Note that for this dual problem to be meaningful, it is necessary
that  $(\mu,\nu)\in\XX.$ This is realized if
$(\mu,\nu)\in\dom\Las$ or equivalently if $\inf\eqref{MK}<\infty.$

Applying the first part of Theorem \ref{T3b}, taking into account
the dual equality (\ref{eq-kanto}) and Proposition \ref{res-02}
gives the following

\begin{lemma}\label{res-17}
 Any $(\pi,\omega)\in\PAB\times \XX^*$ is a solution to \emph{(\MK,$\Kb$)} if and only if
\begin{equation}\label{eq-17}
\left\{\begin{array}{ll}
  (a) & \pi\in\Pcmn;\\
  (b) & \pi\in \partial_{\Uc'}\Fb(\eta) \textsl{ where } \\
  (c) & \eta=T^\ast \omega. \\
\end{array}\right.
\end{equation}
\end{lemma}

As $\Fs$ and $\Fb$ are mutually convex conjugates, (\ref{eq-17}-b)
is equivalent to
\begin{equation}\label{eq-122}
    \eta\in \partial_{\Uc^{\prime *}}\Fs(\pi)
\end{equation}
and also equivalent to Fenchel's identity
\begin{equation}\label{eq-122b}
    \Fs(\pi)+\Fb(\eta)=\langle\eta,\pi\rangle
\end{equation}
and by Proposition \ref{res-02} this is also equivalent to
\begin{equation*}
    \left\{\begin{array}{l}
      \Fb(\eta)= 0\\
      \langle\eta,\pi\rangle=\IAB cd\pi.\\
    \end{array}\right.
\end{equation*}
Therefore, with Theorem \ref{T3b} one obtains the

\begin{theorem}\label{res-20}\
\begin{enumerate}
    \item Any $(\pi,\omega)\in \Pcmn\times \XX^*$ is a solution to \emph{(\MK,$\Kb$)} if and only if
\begin{equation}\label{eq-93}
     \left\{\begin{array}{ll}
    (a) & \omega_A\oplus\omega_B\in\overline{\Gamma}  \\
    (b) & \langle\omega,(\mu,\nu)\rangle=\IAB cd\pi;\\
    \end{array}\right.
\end{equation}

    \item Suppose in addition that
\begin{equation}\label{eq-20}
    (\mu,\nu)\in \diffdom \Las.
\end{equation}
Then, \emph{(\MK,$\Kb$)} admits a solution in $\Pcmn\times\XX^*.$
    \item Writing $\eta=\omega_A\oplus\omega_B,$
(\ref{eq-17}-b), (\ref{eq-122}), (\ref{eq-122b}) and (\ref{eq-93})
are equivalent statements.
\end{enumerate}
\end{theorem}

One was allowed to apply the second part of Theorem \ref{T3b}
under the constraint qualification (\ref{xCQbis})=(\ref{eq-20}).
Let us give some details about this abstract requirement.

\begin{remark}\label{rem-04}
Note that if $\AB$ is an uncountable set,
$(\mu,\nu)\not\in\icordom\Las$ for all $\mu\in\PA,$ $\nu\in\PB.$
Indeed, for all $\pi\in \Pcmn$ one can find $(a_o,b_o)$ such that
with $\delta_{(a_o,b_o)}$ the Dirac measure at $(a_o,b_o),$
$\ell_t:=t\delta_{(a_o,b_o)}+(1-t)\pi\not\geq 0$ for all $t<0,$ so
that $\Fs_1(\ell_t)=+\infty$ (Proposition \ref{res-02}-a). This
shows that $[\ell_0,\ell_1]=[\pi,\delta_{(a_o,b_o)}]\subset
\dom\Fs_1$ while $\ell_t\not\in \dom\Fs_1$ for all $t<0.$ Hence,
$(\mu,\nu)\not\in\icordom\Las$ and one has to consider the
assumption (\ref{eq-20}) on $(\mu,\nu)$ rather than
$(\mu,\nu)\in\icordom\Las.$
\end{remark}

\begin{remarks}[Some remarks about $\Fs(\ell),$ $\Fs(|\ell|),$ $\Las(k)$ and
$\Las(|k|)$] Remark \ref{rem-04} shows that $\icordom\Las$ is
empty in general. The following remarks are motivated by the
problem of circumventing this restrictive property which stops us
from applying Theorem \ref{res-20} with an easy sufficient
condition for (\ref{eq-20}).
\begin{enumerate}[(a)]
    \item As $\Uc'$ is a Riesz space (it is the topological
    dual space of the Riesz space $C_{|\s}$), any $\ell$ in $\Uc'$
    admits an absolute value $|\ell|=\ell_++\ell_-$ and one can
    consider the convex and real-valued function
    $\Fs(|\ell|)=\Fs(\ell_+)+\Fs(\ell_-)$ on $\Uc'.$
    \item Rather than the positively homogeneous sublinear
    function $\Las(k)$ one could think of a \emph{real-valued}
    positively homogeneous function of the type $\Las(|k|),$ since its $\icordom$ is nonempty. But,
    unlike $\Uc',$ $\XX$ is not a Riesz space and  $\Las(|k|)$ is
    meaningless.
    \item The dual equality is $\Las(k)=\inf\{\Fs(\ell); \ell\in \Uc',
    T_o\ell=k\}$ and unlike $\Fs,$ the effective domain of
    $\Fs(|\ell|)$ is the whole space. It is natural to think of the function
    $J(k):=\inf\{\Fs(|\ell|); \ell\in \Uc', T_o\ell=k\}$ instead
    of $\Las.$   The corresponding dual equality is
    $J=\iota^*_{\Upsilon_o}$ where $\Upsilon_o:=\{( f', g')\in (C_A\oplus C_B)_{|\s};
    -c\le { f'\oplus g'}\le c\}.$ It follows that $J$ is a positively homogeneous sublinear
    function. But it is not true that $J$ and $\Las$ match on
    $\dom\Las.$ We have $J\le\Las$ and this inequality can be strict. To see this, let us consider the following example.
    Take $A=\{a,\alpha\},$ $B=\{b,\beta\},$
    $c(a,b)=c(a,\beta)=c(\alpha,b)=0$ and $c(\alpha,\beta)=1.$
    Clearly $\Las(\delta_\alpha,\delta_\beta)=c(\alpha,\beta)=1$
    while $J(\delta_\alpha,\delta_\beta)=\Fs(|\delta_{(a,\beta)}+\delta_{(\alpha,b)}-\delta_{(a,b)}|)=0.$
\end{enumerate}
\end{remarks}

Theorem \ref{res-20} is the core of the extended saddle-point
method applied to the Monge-Kantorovich problem. To prove a
practical optimality criterion one still has to translate these
abstract properties.

\section{A sufficient condition of optimality}\label{sec-02}

The aim of this section is to prove Theorem \ref{res-05}.

\subsection{Any finite strongly $c$-cyclically monotone plan is optimal}
Next lemma gives a characterization of the closed convex hull
$\cl\cv A$ of a set $A$ in terms of its support functional
$\iota^*_A.$
\begin{lemma}\label{res-63}
Let $X$ and $Y$ be two topological vector spaces in duality. For
any subset $A$ of $X,$ one has
\begin{equation*}
    x\in\cl\cv(A)\Leftrightarrow \langle x,y\rangle\le
    \iota_A^*(y),\ \forall y\in Y
\end{equation*}
where $\iota_A^*(y)=\sup_{z\in A}\langle z,y\rangle,$ $y\in Y.$
\end{lemma}

\begin{proof}
The biconjugate $\iota_A^{**}$ of the indicator function $\iota_A$
is its closed convex envelope which is also the indicator function
$\iota_{\cl\cv A}$ of $\cl\cv A.$ Therefore,
\begin{eqnarray*}
  x\in\cl\cv(A)
  &\Leftrightarrow&  \iota_A^{**}(x)=0\\
  &\Leftrightarrow&  \iota_A^{**}(x)\le0 \\
  &\Leftrightarrow& \langle x,y\rangle\le
    \iota_A^*(y),\ \forall y\in Y
\end{eqnarray*}
where the second equivalence follows from $ \iota_A^{**}(x)=\sup_y
\langle x,y\rangle-\iota_A^*(y)\ge \langle
x,0\rangle-\iota_A^*(0)=0,$ for all $x\in X.$
\end{proof}

\begin{proposition}\label{res-62}
Any finite strongly $c$-cyclically monotone plan is optimal.
\end{proposition}

This result is a restatement of \cite[Theorem 2]{ST09}, see
Theorem \ref{res-75}-b.

\begin{proof}
This is a corollary of the first (easy) part of Theorem
\ref{res-20}.
\\
Let $\pi$ be a finite strongly $c$-cyclically monotone plan:
There exist two measurable $[-\infty,\infty)$-valued functions $f$
and $g$ such that $\fog$ satisfies \eqref{eq-75}:
\begin{equation*}
    \left\{\begin{array}{ll}
       f\oplus g\leq c &\  \textrm{everywhere}\\
        f\oplus g= c &\ \pi\textrm{-almost everywhere}.\\
    \end{array}\right.
\end{equation*}
Let us first check that $(\mu,\nu)$ belongs to $\XX.$ Indeed,
$|(\mu,\nu)|^*_\La=\inf\{\|\ell\|^*_c;
\ell\in\Uc',\ell_A=\mu,\ell_B=\nu\}\le\|\pi\|^*_c\le\int
c\,d\pi<\infty.$ Let $E_o$ be the vector subspace of $\XX$ spanned
by $(\mu,\nu)$ and by means of Lemma \ref{res-03}, define the
linear form on $E_o$

\begin{equation*}
    \langle\omega_o,t(\mu,\nu)\rangle:=t\int \fog\,d\mnc,
    \quad t\in\R.
\end{equation*}
As $  f\oplus g= c,$ $\pi$-almost everywhere and $\pi\in \Pcmn$,
we have
\begin{equation}\label{eq-94}
     \langle\omega_o,(\mu,\nu)\rangle=\int c\,d\pi.
\end{equation}
On the other hand, thanks to (\ref{xped}) and Theorem
\ref{res-MK}-(1), $\Las(\mu,\nu) = \sup(\overline{K}) \ge
\int\fog\,d\mnc =\langle\omega_o,(\mu,\nu)\rangle.$  It follows
that
$$
\langle\omega_o,k\rangle\leq\Las(k),\quad \forall k\in E_o.
$$
Denoting
 $$
    \Upsilon:=\{( f', g')\in (C_A\oplus C_B)_{|\s};
{ f'\oplus g'}\le c\},
 $$
one has $\Las(k)=\iota_\Upsilon^*(k),$ $k\in\XX.$ In particular,
it is a $[0,\infty]$-valued positively homogeneous convex function
on $\XX.$ By the analytic form of the Hahn-Banach theorem (see
Remark \ref{rem-61} below), there exists an extension $\omega$ of
$\omega_o$ to $\XX$ which satisfies
\begin{equation*}
    \langle\omega,k\rangle\leq\iota_\Upsilon^*(k),\ \forall k\in \XX.
\end{equation*}
By Lemma \ref{res-63}, this means that $\omega$ belongs to the
$\sigma(\Xs,\XX)$-closure $\overline{\Upsilon}$ of $\Upsilon.$ It
is clear that $T_o^*\Upsilon\subset\Gamma$ and one sees that
$T^\ast\overline{\Upsilon}\subset\overline{\Gamma}$ because of the
$\sigma(\XX^*,\XX)$-$\sigma(\LL^*,\LL)$-continuity of $T^\ast:
\XX^*\to\LL^*,$ see \cite[Lemma 4.13]{Leo07a}.  Since $
T^\ast\omega= \omega_A\oplus\omega_B,$ we have
\begin{equation*}
   \omega_A\oplus\omega_B\in \overline{\Gamma}.
\end{equation*}
Together with (\ref{eq-94}), this allows us to apply Theorem
\ref{res-20}-(1) to obtain the desired result.
\end{proof}

\begin{remark}\label{rem-61}
We have used an unusual form of the Hahn-Banach theorem
    where the positively homogeneous sublinear function $\Las(k)$ is
    $(-\infty,+\infty]$-valued instead of real-valued. For a proof of this variant,
    rewrite without any change the proof the analytic version of
    Hahn-Banach theorem based on Zorn's lemma; see for instance
    \cite[Thm. 1.1]{Brezis}.
\\
The reason for not considering this extended version in the
literature might be the following: To obtain separation by
\emph{closed} hyperplane, the positively homogeneous sublinear
function  of interest is the gauge of an \emph{open} convex
neighborhood of zero which is finite.
\end{remark}

\subsection{Measurability considerations. Strong $c$-cyclical
monotonicity revisited}

\newcommand{\Lum}{L_1(\mu)}
\newcommand{\Lun}{L_1(\nu)}
\newcommand{\mon}{\mu\!\otimes\!\nu}

\begin{definitions}
Let $\gamma$ be a probability measure and $\Gamma$ be a set of
probability measures on some measurable space.
\begin{enumerate}[(1)]
      \item  A set $N$  is said to
be $\gamma$-negligible if it is measurable and $\gamma(N)=0.$
     \item A set is said to be $\Gamma$-negligible if it is
$\gamma$-negligible for all $\gamma\in \Gamma.$
    \item A property holds
$\Gamma$-almost everywhere if it holds everywhere except on a
$\Gamma$-negligible set.
     \item  A function $h$  is said to
be $\gamma$-measurable if there exists a $\gamma$-negligible set
$N$ such that $\mathbf{1}_{N^c}h$ is measurable.
    \item  A function $h$  is said to
be $\Gamma$-measurable if there exists a $\Gamma$-negligible set
$N$ such that $\mathbf{1}_{N^c}h$ is measurable.
\end{enumerate}
\end{definitions}

In particular the spaces  $\Lum$ and $\Lun$ of all $\mu$ and
$\nu$-integrable functions on $A$ and $B$ consist of classes with
respect to the $\mu$ and $\nu$-almost everywhere equalities of
$\mu$ and $\nu$-measurable functions.

\begin{definitions}[$\mn$-measurability] These notions are meaningful
only for a measurable product space $\AB.$
\begin{enumerate}[(1)]
    \item A  subset $N$ of $\AB$ is said to be
\emph{$\mn$-negligible} if there exist two measurable sets
$N_A\subset A$ and $N_B\subset B$ such that $\mu(N_A)=\nu(N_B)=0$
and $N\subset (N_A\times B)\cup(A\times N_B).$
    \item A property holds
$\mn$-almost everywhere if it holds everywhere except on a
$\mn$-negligible set.
       \item A function $\varphi$ on $\AB$ is said to be
$\mn$-measurable if there exists a $\mn$-negligible measurable set
$N$ such that $\mathbf{1}_{N^c}\varphi$ is measurable on $\AB.$
\end{enumerate}
\end{definitions}

\begin{proposition}\label{res-04}\
\begin{enumerate}[(1)]
    \item Any $\mn$-negligible set is $\Pmn$-negligible.\\ Hence,
    any $\mn$-measurable function is $\Pmn$-measurable.
    \item
Let $f$ and $g$ be functions on $A$ and $B.$ The following
statements are equivalent.
\begin{enumerate}[(a)]
    \item $f$ is $\mu$-measurable  and $g$ is $\nu$-measurable;
    \item $\fog$ is $\mn$-measurable;
    \item $\fog$ is $\Pmn$-measurable;
    \item $\fog$ is $\mon$-measurable.
\end{enumerate}
\end{enumerate}
\end{proposition}

\begin{proof}

\boulette{(1)} Let $\pi\in \Pmn$ and $N=(N_A\ttimes B) \cup
(A\ttimes N_B)$ be a $\mn$-negligible measurable set. We have
$\pi(N)\le \pi(N_A\ttimes B)+\pi(A\ttimes
N_B)=\mu(N_A)+\nu(N_B)=0.$

Let us prove (2). \\ \boulette{$(a)\Rightarrow (b)$} Let $N_A$ and
$N_B$ be negligible sets such that $f\mathbf{1}_{N_A^c}$ and
$g\mathbf{1}_{N_B^c}$ are measurable. Of course, $N=(N_A\ttimes B)
\cup (A\ttimes N_B)$ is $\mn$-negligible and
$\mathbf{1}_{N^c}\fog$ is measurable.

\Boulette{$(b)\Rightarrow (c)$} This follows from (1).

\Boulette{$(c)\Rightarrow (d)$} Immediate.

\Boulette{$(d)\Rightarrow (a)$} Let $S$ be a measurable subset of
$\AB$ such that $\mon(S)=1$ and $\mathbf{1}_S\fog$ is measurable.
For all $a\in A,$ denote $S_a=\{b\in B; (a,b)\in S\}$ the
$a$-section of $S.$ It is measurable and by Fubini's theorem
\begin{equation*}
    1=\mon(S)=\int_A\left[\int_B\mathbf{1}_S(a,b)\,\nu(db)\right]\,\mu(da)
    =\int_A\nu(S_a)\,\mu(da).
\end{equation*}
Therefore, $\nu(S_a)=1,$ $\mu$-a.e.\! and there exists some
$a_o\in A$ such that $\nu(S_{a_o})=1.$ As a section of a
measurable function, the function $b\mapsto
\mathbf{1}_{S_{a_o}}(b)(f(a_o)+g(b))$ is measurable. It follows
that $\mathbf{1}_{S_{a_o}}g$ is measurable: this proves that $g$
is $\nu$-measurable. A similar proof works for $f$.
\end{proof}

It is immediate from the Definition \ref{def-04} that a transport
plan $\pi\in \Pmn$ is  strongly $c$-cyclically monotone if and
only if there exist a $\mu$-measurable function $ f$ on $A$ and a
$\nu$-measurable function
    $ g$ on  $B$  such that
    \begin{equation}\label{eq-07}
    \left\{\begin{array}{l}
       f\oplus g\leq c\quad \mn\textrm{-almost everywhere} \\
        f\oplus g= c\quad \pi\textrm{-almost everywhere.} \\
    \end{array}\right.
    \end{equation}
The underlying measurability properties of $\fog$ which are
required by \eqref{eq-07} are insured by Proposition \ref{res-04}.

By Proposition \ref{res-04}-1, if a property holds true
$\mn$-almost everywhere, then it is still true $\Pmn$-almost
everywhere and a fortiori $\Pcmn$-almost everywhere. Therefore,
\begin{equation}\label{eq-08}
    \left\{\begin{array}{l}
       f\oplus g\leq c\quad \Pcmn\textrm{-almost everywhere} \\
        f\oplus g= c\quad \pi\textrm{-almost everywhere.} \\
    \end{array}\right.
    \end{equation}
is weaker than the strong $c$-cyclical monotonicity. Without
changing a word to the proof of Proposition \ref{res-62}, one
obtains the following sufficient \textit{condition} of optimality.

\noindent\textbf{Theorem \ref{res-05}.} \textit{Let $\pi\in\Pcmn$
be any finite plan. If there exist a $\mu$-measurable function $
f$ on $A$ and a $\nu$-measurable function $ g$ on  $B$  which
satisfy \eqref{eq-08}, then $\pi$ is optimal.}

\subsection{The Counterexamples \ref{res-06}-(d,e)}\label{sec-ctex}

They are optimal plans which are not  strongly $c$-cyclically
monotone but they both satisfy the weaker property \eqref{eq-08}.

\begin{enumerate}[-]
    \item Counterexample \ref{res-06}-(d). Let $A=B=[0,1]$ both
    equipped with the Lebesgue measure $\lambda=\mu=\nu.$ Define
    $c$ to be $\infty$ above the diagonal and $1-\sqrt{a-b}$ for
    $b\le a.$ The set $\Pcmn$ is reduced to the uniform
    probability measure $\pi$ on the diagonal which is a fortiori optimal. It is shown in
    \cite{BGMS08} that $\pi$ is not strongly $c$-cyclically
    monotone. But \eqref{eq-08} is trivially satisfied.

    \item Counterexample \ref{res-06}-(e). Let
    $A=B=\mathbb{N}\cup\{\omega\}$ where $\omega$ is a ``number"
    larger than all $n\in\mathbb{N}.$ Equip $A$ and $B$ with the
    discrete topology and define $\mu=\nu$ with a full support.
    Define $c(a,b)=\left\{\begin{array}{ll}
      \infty & \textrm{for }a<b \\
      1 & \textrm{for }a=b \\
      0 & \textrm{for }a>b \\
    \end{array}\right.$ for each $a\in A$ and $b\in B.$ Again, the set $\Pcmn$ is reduced to a
    single probability measure $\pi$ which is a fortiori optimal
    and \eqref{eq-08} is trivially satisfied. Nevertheless, it is
    proved in \cite{BS09} that $\pi$ is not strongly $c$-cyclically
    monotone.
\end{enumerate}

\section{A necessary condition of optimality}\label{sec-01}

The aim of this section is to prove Theorem \ref{res-73} which
states that an optimal plan satisfies approximately \eqref{eq-08}.

\subsection{A first approach to the necessary condition}
We sketch a direct approach to the necessary condition and
emphasize some problems which remain to be solved. The Kantorovich
dual equality (\ref{eq-kanto}) yields a maximizing sequence
$\{(f_n,g_n)\}_{n\ge1}$ in $C_A\!\times\!C_B.$ Assume that $\inf
\eqref{MK}<\infty$ and define $c_n:=f_n\oplus g_n$ so that $c_n\le
c$ and $\Lim n\int c_n\,d\pi^*=\inf \eqref{MK}=\int c\,d\pi^*$ for
any optimal plan $\pi^*.$ Clearly, $c-c_n\ge0$ and $\Lim n \int
(c-c_n)\,d\pi^*=0.$ Hence, $c_n$ converges to $c$ in $L^1(\pi^*)$
and one can extract a subsequence, denoted $\seq ck,$ which
converges  to $c$ pointwise $\pi^*$-almost everywhere. By a result
of Borwein and Lewis \cite[Corollary 3.4]{BL92}, $\seq ck$
converges pointwise $\pi^*$-almost everywhere to some sum function
$f\oplus g.$ Therefore,
$
    c=f\oplus g,\ \pi^*\textrm{-almost everywhere.}
$

The remaining problem of extending $f$ and $g$ such that $f\oplus
g\le c$ everywhere is not obvious. By Tykhonov's theorem, one can
extract a \emph{subnet} from $\seq ck$ which converges pointwise
to a $[-\infty,+\infty]$-valued function $\tilde{c}$ such that
$\tilde{c}\le c$ everywhere and $\tilde{c}=c,$ $\pi^*$-almost
everywhere.  Unfortunately, a subnet limit is not enough to insure
that $\tilde{c}$ is measurable. In addition, one cannot apply the
above cited Borwein-Lewis convergence result since
$[-\infty,+\infty]$ is not a group. Consequently, one cannot
assert that $\tilde{c}$ is of the sum form $f\oplus g.$

\subsection{A necessary condition}

The idea in this section is to approximate $\Phi=\iota_{\Gamma}$
by the sequence $\seq{\Phi}k$ with $\Phi_k=\iota_{\Gamma_k}$ the
convex indicator function of
\begin{equation*}
    \Gamma_k:=\{u\in\Uo; -kc\le u\le c\},\quad k\ge1.
\end{equation*}
We define
\begin{equation}\label{eq-01}
\Fs_k(\ell)=\Ncs{\ell^+}+k\Ncs{\ell^-},\quad\ell\in\Uc'
\end{equation}
its convex conjugate, where $\Ncs{\ell}=\sup\{\langle
u,\ell\rangle; u\in\Uc: \Nc u\le1\}$ is the dual norm on $\Uc'.$
Note that as $\Uc'$ is a Riesz space, the positive and negative
part $\ell^+$ and $\ell^-$ are meaningful for all $\ell\in\Uc'.$
\\
The restriction of $\Fs_k$ to the set $\Qcmn$ of all $q\in\Uc'$
with first and second marginals $\mu$ and $\nu$ is denoted
\begin{equation*}
F_k(q)=\Ncs{q^+}+k\Ncs{q^-},\quad q\in \Qcmn:=\{q\in \Uc';
q_A=\mu,q_B=\nu\}
\end{equation*}
 It is clear that
$\seq Fk$ converges pointwise and increasingly to the restriction
$F$ of $\Fs$ to $\Qcmn.$ By Proposition \ref{res-02}-c, the
effective domain of $F$ is  $\Pcmn.$ Hence, $F$ is the objective
function of \eqref{MK}:
\begin{equation*}
    F(q)=\left\{%
\begin{array}{ll}
     \IAB c\,dq, & \hbox{if }q\in\Pcmn \\
    +\infty, & \hbox{otherwise} \\
\end{array}%
\right.,\quad q\in\Qcmn
\end{equation*}

\begin{lemma}\label{res-69}\
\begin{enumerate}[(a)]
    \item The following $\Gamma$-convergence result  $$\Glim kF_k=F$$ holds true
    for both the $\ast$-weak topology $\sigma(\Qcmn,\Uc)$ and the strong
    topology  associated with $\Nc{\cdot}^*$ on $\Qcmn;$
    \item $\min F=\Lim k \inf F_k=\sup_k \min F_k.$
\end{enumerate}

\end{lemma}

\begin{proof}
Since $\Qcmn$ is $\sigma(\Uc',\Uc)$-closed, one sees with the
Banach-Alaoglu theorem that $F_k$ is
$\sigma(\Qcmn,\Uc)$-inf-compact for all $k\ge1.$ It is a fortiori
\lsc\ with respect to $\sigma(\Qcmn,\Uc)$ and $\Nc{\cdot}^*.$ As
the sequence $\seq Fk$ is increasing, it is also
$\sigma(\Qcmn,\Uc)$-equicoercive.
\\
Now, (a) and (b) follow respectively from \cite[Prop.\! 5.4 \&
Rem.\! 5.5]{DalMaso} and \cite[Thm.\! 7.8]{DalMaso}.
\end{proof}

Let us denote $M_\epsilon(F_k)$ the set of the
$\epsilon$-minimizers of $F_k:$
$$
M_\epsilon(F_k)=\{q\in \Qcmn; F_k(q)\le\inf F_k+\epsilon\}, \quad
\epsilon>0, k\ge1
$$
and $M(F)$ the set of all the minimizers of $F.$

\begin{lemma}\label{res-70}
Clearly, the set of all the optimal plans is $M(F).$
\\
Assume that $\Pcmn$ is nonempty. Then $M(F)$ is nonempty and
\begin{equation*}
    M(F)=\bigcap_{\epsilon>0, k\ge1}\overline{M_\epsilon}(F_k)
\end{equation*}
where $\overline{M_\epsilon}(F_k)$ is the closure of
$M_\epsilon(F_k)$ with respect to $\Nc{\cdot}^*.$
\end{lemma}

\begin{proof}
Lemma \ref{res-69} and \cite[Thm.\! 7.19]{DalMaso} give us
$M(F)=\bigcap_\epsilon K\textrm{-}\Liminf k M_\epsilon(F_k)$ where
$K\textrm{-}\Liminf k M_\epsilon(F_k)$ is the liminf in the sense
of Kuratowski associated with the topology generated by
$\Nc{\cdot}^*.$ Since $\seq Fk$ is increasing, we obtain
$K\textrm{-}\Liminf k
M_\epsilon(F_k)=\bigcap_{k\ge1}\overline{M_\epsilon}(F_k)$ which
completes the proof.
\end{proof}

This result invites us to learn more about $M_\epsilon(F_k).$ We
denote $\Uc''$ the strong bidual of $\Uc$ and denote $\Nc{\cdot}$
its norm. Let $\overline{\Sigma}_k$ be the
$\sigma(\Uc'',\Uc')$-closure of
$$
    \Sigma_k=\{f\oplus g\in(C_A\oplus C_B)_{|\s}; -kc\le f\oplus g\le c\}
$$
and $\Fb_k$ the largest convex $\sigma(\Uc'',\Uc')$-\lsc\
extension of $\Phi_k$ to $\Uc''.$ We denote $\Las_k$ the analogue
of $\Las$ where $\Phi$ is replaced by $\Phi_k.$

\begin{lemma}\label{res-71}
For each $\epsilon>0,$ $k\ge1$ and $q\in M_\epsilon(F_k),$ there
exist $\etab\in\overline{\Sigma}_k,$ $\etat\in\Uc''$ and
$\lt\in\Uc'$ such that $\Ncs{\lt-q}\le\sqrt{\epsilon},$
$\Nc{\etat-\etab}\le\sqrt{\epsilon}$ and
$\lt\in\partial\Fb_k(\etat).$
\end{lemma}

\begin{proof}
Since the effective domain of $\Fs_k$ is the whole space $\Uc',$
the dual equality insures that the effective domain of $\Las_k$ is
the whole space $\XX.$ In particular,
\begin{equation*}
    \partial\Las_k(\mu,\nu)\not=\emptyset.
\end{equation*}
This non-emptiness is crucial. The approximation $F_k$ of $F$ was
introduced to circumvent the problem of knowing whether
$\partial\Las(\mu,\nu)$ is empty or not. One is allowed to apply
Theorem \ref{T3b} and Proposition \ref{res-01}.
\\
Let $\qb$ be a minimizer of $F_k.$ We have $\qb\in \Qcmn$ and
$\partial\Fs_k(\qb)$ is nonempty. More precisely, with $\ob\in
\partial\Las_k(\mu,\nu)$ and $\etab=\ob_A\oplus\ob_B$ we have $
\etab\in\partial\Fs_k(\qb)$ and
\begin{equation}\label{eq-100}
    \langle\etab,\qb\rangle=\langle\etab,q\rangle,\quad\forall
    q\in \Qcmn.
\end{equation}
Fenchel's equality is
$\Fs_k(\qb)=\langle\etab,\qb\rangle-\Fb_k(\etab).$ But,
$\Fb_k(\etab)$ is equal to 0 or $+\infty$ and
$\Fs_k(\qb)>-\infty.$ Hence, $\Fb_k(\etab)=0$ and
\begin{equation*}
    \left\{%
\begin{array}{l}
    \etab\in \overline{\Sigma}_k\\
    \Fs_k(\qb)=\langle\etab,\qb\rangle\\
\end{array}%
\right.
\end{equation*}

As the level sets of $\F_k$ are $\Nc\cdot$-bounded, by Proposition
\ref{res-01}-b' we have $\etab\in\Uc''.$ Now, we consider the
topological duality $\langle\Uc',\Uc''\rangle.$
\\
Let us take $q$ in $M_\epsilon(F_k).$ With (\ref{eq-100}) one sees
that $\langle\etab,\qb\rangle=\langle\etab,q\rangle.$ It follows
that
$\Fs_k(q)\le\Fs_k(\qb)+\epsilon=\langle\etab,q\rangle+\epsilon=\langle\etab,q\rangle-\Fb_k(\etab)+\epsilon.$
This means that $q$ stands in the $\epsilon$-subdifferential
$\partial_\epsilon\Fb_k(\etab)$ of $\Fb_k$ at $\etab.$ One
completes the proof applying Ekeland's principle \cite[Thm.\!
6.2]{Eke-Temam}.
\end{proof}

\begin{lemma}\label{res-72}
Let $\pi$ be an optimal plan. There exists a sequence $\seq{\lt}k$
in $\Uc'$ and two sequences $\seq{\etat}k$ and $\seq{\etab}k$ in
$\Uc''$ such that
\begin{enumerate}[(i)]
    \item $\Lim k\Ncs{\lt_k-\pi}=0$
    \item[and] for each $k\ge1:$
    \item $\etat_k\in\partial\Fs_k(\lt_k);$
    \item $\Nc{\etat_k-\etab_k}\le 1/k$ and $\etab_k\in \overline{\Sigma}_k;$
    \item $\Ncs{\lt_k^-}\le (\inf\eqref{MK}+2)/k.$
\end{enumerate}
\end{lemma}

\newcommand{\ek}{{\epsilon_k}}

\begin{proof}
By Lemma \ref{res-70}, there exists a sequence $\seq q k$ such
that
\begin{equation}\label{eq-103}
    \Lim k \Ncs{q_k-\pi}=0
\end{equation}
and $q_k\in M_\ek(F_k)$ for each $k\ge1,$ where $\Lim k \ek=0.$
\\
By Lemma \ref{res-71}, for each $k\ge1,$ there exist
$\lt_k\in\Uc',$ $\etab\in\overline{\Sigma}_k$ and $\etat_k\in\Uc'$
such that $\Ncs{\lt_k-q_k}\le \sqrt{\ek},$
$\Nc{\etat_k-\etab_k}\le\sqrt{\ek}$ and
$\etat_k\in\partial\Fs_k(\lt_k).$
\\
Since $\Fs_k$ is $k$-Lipschitz, we have
\begin{eqnarray*}
  \Fs_k(\lt_k) &\le& \Fs_k(q_k)+k\Nc{\lt_k-q_k} \\
  &\le& \inf\eqref{MK} +\ek+k\sqrt{\ek} \\
  &\le&  \inf\eqref{MK} + 2
\end{eqnarray*}
taking
$$\ek=1/k^2.$$
It follows from (\ref{eq-01}) that
$\Ncs{\lt_k^-}\le\Fs_k(\lt_k)/k\le (\inf\eqref{MK} + 2)/k.$ One
concludes with $\Ncs{\lt_k-\pi}\le
\Ncs{\lt_k-q_k}+\Ncs{q_k-\pi}\le 1/k+\Ncs{q_k-\pi}$ and
(\ref{eq-103}).
\end{proof}

\begin{lemma}\label{res-67}
Let $\eta\in\Uc''$ be in the $\sigma(\Uc'',\Uc^{\prime})$-closure
of $\Sigma_k$ for some $k\ge1.$  For each probability measure
$p\in\Pc,$ there exists a function $\etah$ in $L_1(p)$ such that
\begin{equation}\label{eq-02}
\left\{\begin{array}{ll}
    (a) &  \langle \eta,h.p \rangle=\IAB h\etah\,dp,\quad \forall h\in L_\infty(p)\\
    (b) & \etah\in \cl_{L_1(p)}(\Sigma_k)
\end{array}\right.
\end{equation}
where $\cl_{L_1(p)}(\Sigma_k)$ is the closure of $\Sigma_k$ in
$L_1(p)$ with respect to its usual strong topology.
\end{lemma}

\begin{proof}
For any bounded measurable function $h$ on $\AB$ we have
\begin{equation}\label{eq-98}
    |\langle\eta,h.p\rangle|\le\|\eta\|_{\Uc''}\IAB |h|c\,dp
\end{equation}
since $|\langle \eta,h.p\rangle|\le\|\eta\|_{\Uc''}\|h.p\|_{\Uc'}$
with $\|h.p\|_{\Uc'}=\IAB |h|c\,dp.$ Let $\seq hk$ be a sequence
of bounded measurable functions such that the sequence
$(|h_k|)_{k\ge1}$ decreases pointwise to zero. By (\ref{eq-98})
and dominated convergence, we have $\Lim k \langle
\eta,h_k.p\rangle=0.$ This means that $h\mapsto \langle
\eta,h.p\rangle$ is $\sigma$-additive: there exists a measure
$\rho$ such that $\langle \eta,h.p\rangle=\IAB h\,d\rho,$ $\forall
h.$ By (\ref{eq-98}), $\rho$ is absolutely continuous with respect
to $p:$ there exists $\etah$ in $L_1(p)$ such that (\ref{eq-02}-a)
holds.
\\
Let $\{u_\alpha\}_\alpha$ be a net in $\Sigma_k$ such that
$\lim_\alpha u_\alpha=\eta.$ For all $h$ in $L_\infty(p),$
\begin{equation*}
    \IAB h\etah\,dp=\langle\eta,h.p\rangle=\lim_\alpha \langle
u_\alpha,h.p\rangle=\lim_\alpha\IAB hu_\alpha\,dp.
\end{equation*}
This means that $\etah$ is in the
$\sigma(L_1(p),L_\infty(p))$-closure of $\Sigma_k.$ As $\Sigma_k$
is convex, this implies (\ref{eq-02}-b).
\end{proof}

We are now ready to give the proof of Theorem  \ref{res-73}.

\noindent\textbf{Theorem \ref{res-73}.} \textit{Let $\pi$ be any
optimal plan and let us take $\epsilon>0$ and $p\in\Pc.$ Then,
there exist $\varphi\in L_1(\pi+p),$ $f\in C_A,$ $g\in C_B$ and a
measurable subset $D\subset\AB$ such that
\begin{enumerate}
    \item
    $\varphi= c,\ \mathbf{1}_{(\s\setminus D)}.\pi$-almost
    everywhere;
    \item
    $\int_{D} (1+c)\,d\pi\le\epsilon;$
    \item
    $-c/\epsilon\le\varphi\le c,\ (\pi+p)$-almost everywhere;
    \item
    $-c/\epsilon\le f\oplus g\le c,$ everywhere;
    \item
    $\|\varphi-f\oplus g\|_{L_1(\pi+p)}\le\epsilon.$
\end{enumerate}}

\begin{proof}
By Lemma \ref{res-72}, for any $\epsilon>0,$ there exist $k$ large
enough, $\lt$ in $\Uc'$ and $\etat,\etab$ in $\Uc''$ such that
\begin{equation*}
\left\{\begin{array}{ll}
    \textrm{(i)}&\Ncs{\lt-\pi}\le\epsilon \textrm{ and }\Ncs{\lt^-}\le\epsilon \\
    \textrm{(ii)}&\etat\in\partial\Fs_k(\lt)\\
    \textrm{(iii)}&\etab\in\overline{\Sigma}_k\\
    \textrm{(iv)} &\Nc{\etat-\etab}\le \epsilon \\
\end{array}\right.
\end{equation*}
 Let $\widehat{\etat}$ and $\widehat{\etab}$ be the
functions in $L_1(\pi+p )$ which are built from $\etat$ and
$\etab$ as in Lemma \ref{res-67} and satisfy: $\langle
\etat,h.(\pi+p ) \rangle=\IAB h\widehat{\etat}\,d(\pi+p )$ and
$\langle \etab,h.(\pi+p ) \rangle=\IAB h\widehat{\etab}\,d(\pi+p
)$ for all $h\in L_\infty(\pi+p ).$ By (iii) and (\ref{eq-02}-b),
there exist $f\in C_A$ and $g\in C_B$ such that $-kc\le f\oplus
g\le c$ and
$$
\|\widehat{\etab}-f\oplus g\|_{L_1(\pi+p)}\le\epsilon.
$$
With the notation of Lemma \ref{res-67}, we have
$\|\etah\|_{L_1(p)}\le [\IAB c\,dp] \Nc\eta$ for all
$\eta\in\Uc''.$ Taking $\varphi=\widehat{\etat},$ this and (iv)
give $ \|\varphi-\widehat{\etab}\|_{L_1(\pi+p)}\le\epsilon\IAB
c\,d(\pi+p).$ Therefore,
$$
\|\varphi-f\oplus g\|_{L_1(\pi+p)}\le\left(1+\IAB
c\,d(\pi+p)\right)\epsilon.
$$
By (ii), we have $\etat\in \dom\Fb_k=\Sigma_k.$ Hence,
$-kc\le\varphi=\widehat{\etat}\le c,$ $(\pi+p)$-a.e.

Since $\Uo$ is identified with $\Cs$ which is a subspace of
$C_{\s},$ see (\ref{eq-03}), any $\ell\in\Uc'$ is the restriction
to $\Cs$ of a continuous linear form on $C_\s.$ Therefore it can
be uniquely decomposed as $\ell=\ell^a+\ell^s$ where $\ell^a$ is a
measure and $\ell^s$ is singular with
$\Ncs{\ell}=\Ncs{\ell^a}+\Ncs{\ell^s}.$ By (i), we have
$\Ncs{\lt^+-\pi}\le2\epsilon.$ Since $\pi$ is a measure, its
singular part vanishes: $\pi^s=0.$ It follows  that
$\Ncs{\lt^{+a}-\pi}\le2\epsilon$ where  $\lt^{+a}$ is the measure
part of $\lt^+.$ Therefore, the nonnegative measures $\lt^{+a}$
and $\pi$ are concentrated on the same set except for a measurable
set $D$ such that $\int_Dc\,d(\lt^{+a}+\pi)\le2\epsilon.$\
\\
    Finally, (ii) implies that
    $
    c=\varphi,
    $ $\lt^{+a}$-almost everywhere
(and also that $-kc=\varphi,$ $\lt^{-a}$-almost everywhere). One
completes the proof, putting everything together and replacing
$\epsilon$ by $\epsilon/C$ with an appropriate constant $C.$
\end{proof}

%\bibliographystyle{plain}
%\bibliography{bib-christian}

\end{document}